\documentclass[12pt]{article}
\usepackage[top=2cm, bottom=2cm, right=2cm, left=2cm]{geometry}
\usepackage{amsfonts}
\usepackage{amssymb}
\usepackage{amsmath}
\usepackage{amsthm}
\usepackage{amsfonts}
\usepackage{graphicx}
\usepackage{enumerate}
\usepackage{hyperref}%

\newtheorem{example}{Example}
\newtheorem{definition}{Definition}
\newtheorem{corollary}{Corollary}

\newtheorem{proposition}{Proposition}

\newtheorem{remark}{Remark}

\newtheorem{theorem}{Theorem}
\newtheorem{lemma}{Lemma}
\begin{document}

\title{Global controllability properties of linear control systems}
\date{}
\author{Fritz Colonius \\
Institut
für Mathematik, Augsburg Universität,
Germany \and Alexandre J. Santana\thanks{
Partially supported by CNPq grant n. 309409/2023-3} \\
Department of Mathematics, State University of Maringá,
 Brazil}
\maketitle

\begin{abstract}For linear control systems with bounded control range, the
state space is compactified using the Poincar\'{e} sphere. The linearization
of the induced control flow allows the construction of invariant manifolds on
the sphere and of corresponding manifolds in the state space of the linear
control system.
\end{abstract}

\newline
\noindent {}\noindent \textit{AMS 2020 subject classification}: 93B05, 93C05, 37B20

\section{Introduction\label{Section1}}

We study global controllability properties of linear control systems on
$\mathbb{R}^{n}$ with control restrictions of the form%
\begin{equation}
\dot{x}(t)=Ax(t)+Bu(t),\,u\in\mathcal{U}, \label{linear}%
\end{equation}
where $A\in\mathbb{R}^{n\times n},B\in\mathbb{R}^{n\times m}$, and the set of
control functions is defined by%
\begin{equation}
\mathcal{U}=\{u\in L^{\infty}(\mathbb{R},\mathbb{R}^{m})\left\vert u(t)\in
U\text{ for almost all }t\in\mathbb{R}\right.  \}; \label{U}%
\end{equation}
here the control range $U$ is a bounded neighborhood of $0\in\mathbb{R}^{m}$.
We denote the solution for initial condition $x(0)=x_{0}\in\mathbb{R}^{n}$ and
control $u\in\mathcal{U}$ by $\varphi(t,x_{0},u),t\in\mathbb{R}$.

If the Kalman rank condition $\mathrm{rank}[B~AB\cdots A^{n-1}B]=n$ holds,
there is a unique maximal set $D$, where complete controllability holds, and
$D$ contains the origin in the interior. The present paper analyzes global
controllability properties based on a compactification of $\mathbb{R}^{n}$ by
a variant of the classical Poincar\'{e} compactification from the theory of
polynomial differential equations.

The basic geometric idea is quite simple. A copy of $\mathbb{R}^{n}$ is
attached to the sphere $\mathbb{S}^{n}$ in $\mathbb{R}^{n+1}$ at the north
pole. Then one takes the central projection in $\mathbb{R}^{n+1}$ to the
northern hemisphere $\mathbb{S}^{n,+}$ of $\mathbb{S}^{n}$ called the
Poincar\'{e} sphere. The equator $\mathbb{S}^{n,0}$ of $\mathbb{S}^{n}$
represents infinity since for points in $\mathbb{R}^{n}$ with $\left\Vert
x_{k}\right\Vert \rightarrow\infty$ the images in $\mathbb{S}^{n}$ approach
$\mathbb{S}^{n,0}$. The local analysis of points on $\mathbb{S}^{n,0}$ allows
us to arrive at conclusions about the behavior of the original system
\textquotedblleft near infinity\textquotedblright. In order to work this out,
the machinery of control flows is helpful: The (open loop) behavior of control
systems is described by a continuous flow $\Phi$ on $\mathcal{U}%
\times\mathbb{R}^{n}$ and tools from dynamical systems theory, in particular,
the Selgrade decomposition can be invoked.

The approach of the present paper is mainly motivated by techniques from
Poin\-car\'{e} compactification in the theory of polynomial differential
equations due to Poin\-car\'{e} \cite{Poin}; cf. Cima and Llibre
\cite{CimL90}, Perko \cite{Perko}, Dumortier, Llibre, and Artes \cite{DumLA},
Llibre and Teruel \cite{LlibT}. We use the smooth structure and not only the
topological properties of the associated flows as in the earlier paper
Colonius, Santana, and Viscovini \cite{ColSV24}. We expect that some of these
techniques will also be useful when applied to affine control systems and to
polynomial control systems.

The theory of control flows, control sets, and chain control sets is developed
in Colonius and Kliemann \cite{ColK00} and Kawan \cite{Kawan13}. For further
contributions we refer to Ayala, da Silva, and Mamani \cite{ASM23}, da Silva
\cite{daS23}\textbf{, }Boarotto and Sigalotti \cite{BoaS20}, Tao, Huang, and
Chen \cite{TaoHC}. Cannarsa and Sigalotti \cite{CanS21} show that approximate
controllability for bilinear control systems is equivalent to exact
controllability. The Selgrade decomposition for linear flows on vector bundles
is due to Selgrade \cite{Selg75}; cf. Salamon and Zehnder \cite{SalZ88} and
Colonius and Kliemann \cite{ColK00, ColK14}. For linear flows with chain
transitive base space, Selgrade's theorem provides a Whitney decomposition of
the vector bundle into subbundles such that the projections of the subbundles
to the projective bundle yield the maximal chain transitive sets of the
induced projective flow; cf. e.g. \cite[Theorem 9.2.5]{ColK14}. We will lift
linear control systems of the form (\ref{linear}) to bilinear control systems
on $\mathbb{R}^{n+1}$ and derive Selgrade decompositions for the lifted linear
control flow $\Phi^{1}$ on $\mathcal{U}\times\mathbb{R}^{n+1}$ as well as for
the linearization of the projected control flow $\pi\Phi^{1}$ on
$\mathcal{U}\times\mathbb{S}^{n}$. Furthermore, corresponding invariant
manifolds in $\mathbb{S}^{n}$ and $\mathbb{R}^{n}$ are constructed. For some
pertinent references concerning invariant manifolds, see the introduction of
Section \ref{Section5}.

The main results of this paper are Corollary \ref{Corollary_limit} clarifying
the limit behavior for time tending to infinity of trajectories on the
Poincar\'{e} sphere $\mathbb{S}^{n}$. Theorem \ref{Theorem_LinLyap} determines
the Lyapunov exponents for the induced control flow and Corollary
\ref{Corollary_stable_man} describes stable manifolds on $\mathbb{S}^{n}$.
Consequences for the original linear control system in $\mathbb{R}^{n}$ are
drawn in Theorem \ref{Theorem_R}.

The contents of this paper are as follows. Section \ref{Section2} contains
preliminary results on control sets, chain control sets, and control flows for
control-affine systems on manifolds. Results from Colonius, Santana, and
Viscovini \cite{ColSV24} for linear control systems are recalled. Section
\ref{Section3} uses the Selgrade decomposition of the lifted control flow
$\Phi^{1}$ on $\mathcal{U}\times\mathbb{R}^{n+1}$ to characterize the limit
behavior of trajectories on the Poincar\'{e} sphere $\mathbb{S}^{n}$ by the
chain control sets. There are one or two chain control sets which are not
subsets of the equator $\mathbb{S}^{n,0}$. The Lyapunov spaces $L(\lambda
_{i})$ of the matrix $A$ yield maximal chain transitive sets $_{\mathbb{S}%
}L(\lambda_{i})^{\infty}$ for the induced flow on the equator $\mathbb{S}%
^{n,0}$. In Section \ref{Section4} the induced control flow $\pi\Phi^{1}$ on
$\mathcal{U}\times\mathbb{S}^{n}$ is linearized. When the base space of the
linearized flow $T\pi\Phi$ is restricted to $\mathcal{U}\times\,_{\mathbb{S}%
}L(\lambda_{i})^{\infty}$, the corresponding Selgrade decomposition and the
Lyapunov exponents are determined.\ Section \ref{Section5} contains results on
stable manifolds and Section \ref{Section6} presents several examples.

\textbf{Notation.} The closure of a set $A$ in a metric space is denoted by
$\overline{A}$. The origin in $\mathbb{R}^{n}$ is $0_{n}$ and $0_{1}$ is
abbreviated by $0$. The projection from $\mathbb{R}_{0}^{n}=\mathbb{R}%
^{n}\setminus\{0_{n}\}$ to the sphere $\mathbb{S}^{n-1}$ is $\pi x=\frac
{x}{\left\Vert x\right\Vert }$ for $x\in\mathbb{R}_{0}^{n}$.

\section{Preliminaries\label{Section2}}

For general nonlinear control systems, control sets, chain control sets, and
control flows are defined and some of their properties are recalled. Then, for
linear control systems, control sets and chain control sets are characterized.

\subsection{Control sets, chain control sets, and control
flows\label{Subsection2.1}}

Consider control-affine systems of the form
\begin{equation}
\dot{x}(t)=X_{0}(x(t))+\sum_{i=1}^{m}u_{i}(t)X_{i}(x(t)),\,u\in\mathcal{U},
\label{control1}%
\end{equation}
where $X_{0},X_{1},\mathbb{\ldots},X_{m}$ are smooth ($C^{\infty}$-)vector
fields on a smooth manifold $M$ and%
\[
\mathcal{U}=\left\{  u\in L^{\infty}(\mathbb{R},\mathbb{R}^{m})\left\vert
u(t)\in U\text{ for almost all }t\in\mathbb{R}\right.  \right\}  ,
\]
where $U$ is a compact convex neighborhood of $0\in\mathbb{R}^{m}$. We assume
that for every control $u\in\mathcal{U}$ and every initial state
$x(0)=x_{0}\in M$ there exists a unique (Carath\'{e}odory) solution
$\varphi(t,x_{0},u),t\in\mathbb{R}$. Background on linear and nonlinear
control systems is provided by the monographs Sontag \cite{Son98}, Jurdjevic
\cite{Jur97}.

For $x\in M$ the controllable set $\mathbf{C}(x)$ and the reachable set
$\mathbf{R}(x)$ are defined as
\begin{align*}
\mathbf{C}(x)  &  =\left\{  y\in M\left\vert \exists u\in\mathcal{U}~\exists
T>0:\varphi(T,y,u)=x\right.  \right\}  ,\\
\mathbf{R}(x)  &  =\left\{  y\in M\left\vert \exists u\in\mathcal{U}~\exists
T>0:y=\varphi(T,x,u)\right.  \right\}  .
\end{align*}
The following definition introduces sets of complete approximate controllability.

\begin{definition}
\label{Definition_control_sets}A nonvoid set $D\subset M$ is called a control
set of system (\ref{control1}) if it has the following properties: (i) for all
$x\in D$ there is a control $u\in\mathcal{U}$ such that $\varphi(t,x,u)\in D$
for all $t\geq0$, (ii) for all $x\in D$ one has $D\subset\overline
{\mathbf{R}(x)}$, and (iii) $D$ is maximal with these properties, that is, if
$D^{\prime}\supset D$ satisfies conditions (i) and (ii), then $D^{\prime}=D$.
\end{definition}

Next we introduce a notion of controllability in infinite time allowing for
(small) jumps between pieces of trajectories. We fix a metric $d$ compatible
with the topology of $M$.

\begin{definition}
\label{intro2:defchains}Let $x,y\in M$. For $\varepsilon,\tau>0$ a controlled
$(\varepsilon,\tau)$\textit{-chain} $\zeta$ from $x$ to $y$ is given by
$k\in\mathbb{N},\ x_{0}=x,x_{1},\ldots,x_{k}=y\in M,\ u_{0},\ldots,u_{k-1}%
\in\mathcal{U}$, and $T_{0},\ldots,T_{k-1}\geq\tau$ with
\[
d(\varphi(T_{j},x_{j},u_{j}),x_{j+1})<\varepsilon\text{ }\,\text{for\thinspace
all}\,\,\,j=0,\ldots,k-1.
\]
If for every $\varepsilon,\tau>0$ there is a controlled $(\varepsilon,\tau
)$-chain from $x$ to $y$, the point $x$ is chain controllable to $y$.
\end{definition}

The chain reachable set is%
\begin{equation}
\mathbf{R}^{c}(x)=\left\{  y\in M\left\vert \forall\varepsilon,\tau
>0~\exists~\text{a controlled }(\varepsilon,\tau)\text{\textit{-chain from }%
}x\text{ to }y\right.  \right\}  . \label{R_c}%
\end{equation}
In analogy to control sets, we define chain control sets as maximal chain
controllable sets.

\begin{definition}
\label{Definition_chain_control}A nonvoid set $E\subset M$ is called a
\textit{chain control set} of system (\ref{control1}) if for all $x,y\in E$
and $\varepsilon,\tau>0$ there is a controlled $(\varepsilon,\tau)$-chain from
$x$ to $y$, and $E$ is maximal with this property.
\end{definition}

The control flow associated with control system (\ref{control1}) is the flow
on $\mathcal{U}\times M$ defined by%
\begin{equation}
\Phi:\mathbb{R}\times\mathcal{U}\times M\rightarrow\mathcal{U}\times
M,\Phi_{t}(u,x)=(\theta_{t}u,\varphi(t,x,u)), \label{control_flow}%
\end{equation}
where $\theta_{t}u=u(t+\cdot)$ is the right shift on $\mathcal{U}$. Note that
$\Phi_{t+\tau}=\Phi_{t}\circ\Phi_{\tau}$ for $t,\tau\in\mathbb{R}$. The space
$\mathcal{U}$ is a compact metrizable space with respect to the weak$^{\ast}$
topology of $L^{\infty}(\mathbb{R},\mathbb{R}^{m})$ (we fix such a metric) and
the shift flow $\theta$ is\ continuous; cf. Kawan \cite[Proposition
1.15]{Kawan13}. Furthermore, also the flow $\Phi$ is continuous; cf.
\cite[Proposition 1.17]{Kawan13}.

A chain transitive set of the control flow is a subset of $\mathcal{U}\times
M$ such that for all $(u,x),(v,y)\in\mathcal{U}\times M$ and all
$\varepsilon,\,\tau>0$ there is an $(\varepsilon,\tau)$-chain $\zeta$ for
$\Phi$ from $(u,x)$ to $(v,y)$ given by $k\in\mathbb{N}\mathbf{,}$
$T_{0},\ldots,T_{k-1}\geq\tau$, and $(u_{0},x_{0})=(u,x),\ldots,(u_{k-1}%
,x_{k-1}),\allowbreak(u_{k},x_{k})=(v,y)\in\mathcal{U}\times M$ with
$d(\Phi(T_{i},u_{i},x_{i}),(u_{i+1},x_{i+1}))<\varepsilon$ for $i=0,\ldots
,k-1$. The relation between chain control sets and the control flow $\Phi$
defined in (\ref{control_flow}) is explained in the following theorem; cf.
Colonius, Santana, and Viscovini \cite[Theorem 2.15]{ColSV24}.

\begin{theorem}
\label{Theorem_equivalence}Let $\mathcal{E}\subset\mathcal{U}\times M$ be a
maximal chain transitive set for the control flow $\Phi$. Then $\left\{  x\in
M\left\vert \exists u\in\mathcal{U}:(u,x)\in\mathcal{E}\right.  \right\}  $ is
a chain control set. Conversely, if $E\subset M$ is a chain control set, then%
\begin{equation}
\mathcal{E}:=\{(u,x)\in\mathcal{U}\times M\left\vert \varphi(t,x,u)\in E\text{
for all }t\in\mathbb{R}\right.  \} \label{lift_E}%
\end{equation}
is a maximal chain transitive set.
\end{theorem}

\subsection{The control sets and chain control sets of linear control systems}

We consider linear control systems of the form (\ref{linear}) and, throughout
the rest of this paper, we assume that the control range $U$ is a compact and
convex neighborhood of the origin. This is not a restriction since for a
bounded neighborhood $U$ of the origin in $\mathbb{R}^{m}$ the trajectories
for controls taking values in the closed convex hull of $U$ can, uniformly on
bounded intervals, be approximated by trajectories for controls with values in
$U$ (cf. Lee and Markus \cite[Theorem 1A, p. 164]{LeeM} for this classical result).

Let $\psi(t,x)=\mathrm{e}^{At}x,t\in\mathbb{R}$, be the linear flow on
$\mathbb{R}^{n}$ generated by $\dot{x}=Ax$. The Lyapunov exponents
$\lambda(x):=\lim_{t\rightarrow\pm\infty}\frac{1}{t}\log\left\Vert
\psi(t,x)\right\Vert $ are equal to the real parts of the eigenvalues of $A$
and the Lyapunov spaces $L(\lambda_{i})$ are the sums of the (real)
generalized eigenspaces for eigenvalues with real parts $\lambda_{i}$. The
following theorem characterizes the Lyapunov spaces of $\dot{x}=Ax$ by the
maximal chain transitive sets of the induced flow on projective space.

\begin{theorem}
\label{thmprojMorse}Let $\mathbb{P}\psi$ be the projection onto $\mathbb{P}%
^{n-1}$ of the linear flow $\psi(t,x)=\mathrm{e}^{At}x$ on $\mathbb{R}^{n}$
and denote the Lyapunov exponents by $\lambda_{1}>\cdots>\lambda_{\ell}%
,1\leq\ell\leq n$.

(i) Then the state space $\mathbb{R}^{n}\mathbb{\ }$can be decomposed into the
Lyapunov spaces $L(\lambda_{i})$,%
\begin{equation}
\mathbb{R}^{n}=L(\lambda_{1})\oplus\cdots\oplus L(\lambda_{\ell}).
\label{Lyap1}%
\end{equation}

(ii) The projections $\mathbb{P}L(\lambda_{i})$ to $\mathbb{P}^{n-1}$ of the
Lyapunov spaces $L(\lambda_{i})$ are the maximal chain transitive sets of the
induced flow $\mathbb{P}\psi$ on $\mathbb{P}^{n-1}$.
\end{theorem}

Assertion (i) is clear. A proof of assertion (ii) is contained in Colonius and
Kliemann \cite[Section 4.1]{ColK14}. A simpler proof can be given, when one
uses that the center Lyapunov space $L(0)$ of $\dot{x}=Ax$ is a maximal chain
transitive set if $0$ is a Lyapunov exponent; cf. Colonius, Santana, and
Viscovini \cite[Theorem 3.2]{ColSV24}. The other Lyapunov spaces
$L(\lambda_{i})$ are maximal chain transitive sets for the shifted systems
$\dot{x}=(A-\lambda_{i}I)x$. Then also the projections $\mathbb{P}%
L(\lambda_{j})$ to $\mathbb{P}^{n-1}$ are chain transitive.

It is convenient to introduce the space $L^{0}$ which coincides with the
Lyapunov space $L(0)$ if $0$ is a Lyapunov exponent and is trivial otherwise.
The space $\mathbb{R}^{n}$ can be decomposed into the sum of the unstable,
center, and stable subspaces of $A$ given by%
\begin{equation}
\mathbb{R}^{n}=L^{+}\oplus L^{0}\oplus L^{-}, \label{decomposition1}%
\end{equation}
where $L^{+}=\bigoplus_{\lambda_{i}>0}L(\lambda_{i})$ and $L^{-}%
=\bigoplus_{\lambda_{i}<0}L(\lambda_{i})$. Denote the associated projections
by $\pi^{0}:\mathbb{R}^{n}\rightarrow L^{0}$ and $\pi^{h}:\mathbb{R}%
^{n}\rightarrow L^{+}\oplus L^{-}$.

Next we characterize the control set and the chain control set. Colonius and
Santana \cite[Corollary 1]{ColSan11} shows that, for every $u\in\mathcal{U}$,
there is a unique bounded solution $e(u,t),t\in\mathbb{R},$ of the
differential equation induced on the hyperbolic subspace $L^{+}\oplus L^{-}$,%
\begin{equation}
\dot{y}(t)=A\pi^{h}y(t)+\pi^{h}Bu(t). \label{hyp}%
\end{equation}
Note that $e(u,t)=e(u(t+\cdot),0)$ for $t\in\mathbb{R}$ and $u\in\mathcal{U}$.

The following theorem characterizes the control set containing $0\in
\mathbb{R}^{n}$ and the unique chain control set of a linear control system.

\begin{theorem}
\label{Theorem_cs}Consider the linear control system (\ref{linear}).

(i) There is a control set $D_{0}$ with $0\in D_{0}$. It is convex and
satisfies%
\[
D_{0}=\left(  \mathbf{C}(0)\cap L^{+}\right)  \oplus\left(  L^{0}%
\cap\operatorname{Im}[B~AB\cdots A^{n-1}B]\right)  \oplus(\overline
{\mathbf{R}(0)}\cap L^{-}).
\]
The sets $\mathbf{C}(0)\cap L^{+}$ and $\overline{\mathbf{R}(0)}\cap L^{-}$
are bounded. If $\operatorname{Im}[B~AB\cdots A^{n-1}B]=\mathbb{R}^{n}$, then
there is a unique control set $D$ with nonvoid interior, and $D=D_{0}$.

(ii) There exists a unique chain control set $E$. It is given by%
\begin{align*}
E  &  =\overline{D_{0}}+L^{0}=(\overline{\mathbf{C}(0)}\cap L^{+})\oplus
L^{0}\oplus(\overline{\mathbf{R}(0)}\cap L^{-})\\
&  =\left\{  e(u,t)+y\left\vert u\in\mathcal{U},t\in\mathbb{R}\text{, and
}y\in L^{0}\right.  \right\}  .
\end{align*}

\end{theorem}

\begin{proof}
Assertion (i) is Colonius, Santana, and Viscovini \cite[Corollary
2.17]{ColSV24}. The first two equalities in (ii) hold by \cite[Theorem
4.8]{ColSV24}. Finally, \cite[Lemma 2.18(iii)]{ColSV24} shows that
$(\overline{\mathbf{C}(0)}\cap L^{+})\oplus(\overline{\mathbf{R}(0)}\cap
L^{-})$ is the chain control set of system (\ref{hyp}) and hence bounded. Thus
it consists of the points $e(u,t)$ with $t\in\mathbb{R}$ and $u\in\mathcal{U}$.
\end{proof}

\section{\textbf{Selgrade decomposition for the lifted control
flow\label{Section3}}}

In this section, we embed linear control systems of the form (\ref{linear})
into bilinear control systems on $\mathbb{R}^{n}\times\mathbb{R}$, which can
be projected to the Poincar\'{e} sphere $\mathbb{S}^{n}\subset\mathbb{R}%
^{n+1}$. The associated control flow $\Phi^{1}$ is a linear flow on the vector
bundle $\mathcal{U}\times\mathbb{R}^{n+1}$. It admits a Selgrade decomposition
into invariant subbundles and we determine the exponential growth rates of
solutions. The subbundles yield the chain control sets of the projected
control system on $\mathbb{S}^{n}$, which are the possible limit sets for time
tending to infinity.

Recall from Colonius, Santana, and Viscovini \cite{ColSV24} the following
construction. Linear control systems of the form (\ref{linear}) on
$\mathbb{R}^{n}$ can be lifted to bilinear control systems with states
$(x(t),x_{n+1}(t))$ in $\mathbb{R}^{n}\times\mathbb{R}=\mathbb{R}^{n+1}$ by%
\begin{equation}
\dot{x}(t)=Ax(t)+x_{n+1}(t)Bu(t),~\dot{x}_{n+1}(t)=0,\text{ }u\in\mathcal{U}.
\label{lift1}%
\end{equation}
The solutions for initial condition $\left(  x(0),x_{n+1}(0)\right)  =\left(
x_{0},r\right)  \in\mathbb{R}^{n}\times\mathbb{R}$, may be written as
\begin{equation}
\varphi^{1}(t,x_{0},r,u)=\left(  e^{At}x_{0}+r\int_{0}^{t}e^{A(t-\sigma
)}Bu(\sigma)d\sigma,r\right)  ,\,t\in\mathbb{R}. \label{phi_1}%
\end{equation}
Observe that, for $r=0$, one has $\varphi^{1}(t,x,0,u)=\left(  \psi
(t,x),0\right)  $ and, for $r=1$, one has $\varphi^{1}(t,x,1,u)=(\varphi
(t,x,u),1)$ for all $t\in\mathbb{R},x\in\mathbb{R}^{n},u\in\mathcal{U}$.
Hence, on the hyperplane $\mathbb{R}^{n}\times\{0\}$ one obtains a copy of the
differential equation $\dot{x}=Ax$ and on the affine hyperplane $\mathbb{R}%
^{n}\times\{1\}$ one obtains a copy of control system (\ref{linear}).

Control system (\ref{lift1}) is a bilinear control system on $\mathbb{R}%
^{n+1}$ which, with $b_{i},i=1,\ldots,m$, denoting the $i$-th column of $B$,
may be written as%
\begin{equation}
\left(
\begin{array}
[c]{c}%
\dot{x}(t)\\
\dot{x}_{n+1}(t)
\end{array}
\right)  =\left[  \left(
\begin{array}
[c]{cc}%
A & 0\\
0 & 0
\end{array}
\right)  +\sum_{i=1}^{m}u_{i}(t)\left(
\begin{array}
[c]{cc}%
0 & b_{i}\\
0 & 0
\end{array}
\right)  \right]  \left(
\begin{array}
[c]{c}%
x(t)\\
x_{n+1}(t)
\end{array}
\right)  . \label{lift2}%
\end{equation}
Define subsets of projective space $\mathbb{P}^{n}$ and of the unit sphere
$\mathbb{S}^{n}$ by%
\begin{align}
\mathbb{P}^{n,0}  &  =\{\mathbb{P}(x,0)\left\vert x\in\mathbb{R}^{n}\right.
\},\,\mathbb{P}^{n,1}=\{\mathbb{P}(x,r)\left\vert x\in\mathbb{R}^{n}%
,r\not =0\right.  \},\label{4.2}\\
\mathbb{S}^{n,+}  &  :=\left\{  (x,r)\in\mathbb{S}^{n}\left\vert
x\in\mathbb{R}^{n},r>0\right.  \right\}  ,\,\mathbb{S}^{n,0}=\{(x,0)\in
\mathbb{S}^{n}\left\vert x\in\mathbb{R}^{n}\right.  \},\nonumber
\end{align}
respectively. Define the map%
\begin{equation}
h^{1}:\mathbb{R}^{n}\rightarrow\mathbb{S}^{n,+},h^{1}(x)=\frac{(x,1)}%
{\left\Vert (x,1)\right\Vert },x\in\mathbb{R}^{n}. \label{h1}%
\end{equation}
Since the northern hemisphere $\mathbb{S}^{n,+}$ can be identified with
$\mathbb{P}^{n,1}$ via $\frac{(x,1)}{\left\Vert (x,1)\right\Vert }%
\sim\mathbb{P}(x,1)$, we also denote the map $x\mapsto\mathbb{P}%
(x,1):\mathbb{R}^{n}\rightarrow\mathbb{P}^{n,1}$ by $h^{1}$. It will be clear
from the context what is meant. Furthermore, we call $\mathbb{P}^{n,0}$ the
projective equator since it is the projection of the equator $\mathbb{S}%
^{n,0}$.

Control system (\ref{lift1}) induces a control flow $\Phi^{1}$ on
$\mathcal{U}\times\mathbb{R}^{n+1}$ defined by%
\begin{equation}
\Phi_{t}^{1}(u,x,r)=(u(t+\cdot),\varphi^{1}(t,x,r,u)),t\in\mathbb{R}%
,(x,r)\mathbb{\in R}^{n+1},\,u\in\mathcal{U}. \label{Fi_1}%
\end{equation}
The maps between the fibers $\{u\}\times\mathbb{R}^{n+1}$ are linear, hence
$\Phi^{1}$ is a linear flow. The projection of system (\ref{lift1}) to
projective space yields a projective control flow $\mathbb{P}\Phi^{1}$ on the
projective Poincar\'{e} bundle $\mathcal{U}\times\mathbb{P}^{n}$. The subsets
$\mathcal{U}\times\mathbb{P}^{n,0}$ and $\mathcal{U}\times\mathbb{P}^{n,1}$
are invariant under the flow $\mathbb{P}\Phi^{1}$. The following proposition
(cf. \cite[Proposition 4.2]{ColSV24}) shows some properties of the flows on
the projective Poincar\'{e} bundle.

\begin{proposition}
\label{Proposition_e}(i) The projectivized flow $\mathbb{P}\psi(t,p),p\in
\mathbb{P}^{n-1}$, of the uncontrolled system $\dot{x}=Ax$ and the flow
$\mathbb{P}\varphi^{1}(t,p,0),p\in\mathbb{P}^{n}$ restricted to the projective
equator $\mathbb{P}^{n,0}\subset\mathbb{P}^{n}$ are conjugate by the
analytical isometry $e_{\mathbb{P}}(p)=\mathbb{P}(x,0)$ for $p=\mathbb{P}%
x\in\mathbb{P}^{n-1}$.

(ii) The map%
\begin{equation}
\left(  \mathrm{id}_{\mathcal{U}},h^{1}\right)  :\mathcal{U}\times
\mathbb{R}^{n}\rightarrow\mathcal{U}\times\mathbb{P}^{n,1},(u,x)\mapsto
(u,\mathbb{P}(x,1)), \label{U_h1}%
\end{equation}
is a conjugacy of the flows $\Phi$ on $\mathcal{U}\times\mathbb{R}^{n}$ and
$\mathbb{P}\Phi^{1}$ restricted to $\mathcal{U}\times\mathbb{P}^{n,1}$.
\end{proposition}

Next note the following lemma.

\begin{lemma}
\label{Lemma_growth1}For $u\in\mathcal{U}$, $y\in L^{0}$, and $r\in\mathbb{R}%
$, the term%
\begin{equation}
e^{At}(-re(u,0)+y)+r\int_{0}^{t}e^{A(t-\sigma)}Bu(\sigma)d\sigma
,t\in\mathbb{R}, \label{3.7}%
\end{equation}
has at most polynomial growth in $t$ and hence the exponential growth rate is%
\[
\lim_{t\rightarrow\pm\infty}\frac{1}{t}\log\left\Vert e^{At}(-re(u,0)+y)+r\int
_{0}^{t}e^{A(t-\sigma)}Bu(\sigma)d\sigma\right\Vert =0.
\]

\end{lemma}

\begin{proof}
Since $e(u,t),t\in\mathbb{R}$, is a solution of (\ref{hyp}) it satisfies
\[
e(u,t)=e^{A\pi^{h}t}e(u,0)+\int_{0}^{t}e^{A(t-\sigma)}\pi^{h}Bu(\sigma
)d\sigma\in L^{+}\oplus L^{-}.
\]
In (\ref{3.7}), we may suppose that the point in $L^{0}$ has the form $ry$
instead of $y$ and compute
\begin{align*}
&  e^{At}(-e(u,0)+y)+\int_{0}^{t}e^{A(t-\sigma)}Bu(\sigma)d\sigma\\
&  =-e^{A\pi^{h}t}e(u,0)+e^{A\pi^{0}t}y+\int_{0}^{t}e^{A(t-\sigma)}\pi
^{h}Bu(\sigma)d\sigma+\int_{0}^{t}e^{A(t-\sigma)}\pi^{0}Bu(\sigma)d\sigma\\
&  =-e(u,t)+e^{A\pi^{0}t}y+\int_{0}^{t}e^{A\pi^{0}\sigma}Bu(t-\sigma)d\sigma.
\end{align*}
Here $e(u,t)$ is bounded and $e^{A\pi^{0}t}$ has at most polynomial growth,
and hence also $\int_{0}^{t}e^{A\pi^{0}\sigma}Bu(t-\sigma)d\sigma$ has at most
polynomial growth. This implies the claim.
\end{proof}

The following theorem determines the Selgrade decomposition of the associated
linear control flow on $\mathcal{U}\times\mathbb{R}^{n+1}$. Denote
$L(\lambda_{i})^{\infty}:=L(\lambda_{i})\times\{0\}\subset\mathbb{R}^{n+1}$.

\begin{theorem}
\label{Theorem_Selg_lift}Consider the linear control flow $\Phi^{1}$ on
$\mathcal{U}\times\mathbb{R}^{n+1}$ associated to the lift (\ref{lift2}) of a
linear control system of the form (\ref{linear}).

(i) Then $\Phi^{1}$ has the Selgrade decomposition%
\begin{equation}
\mathcal{U}\times\mathbb{R}^{n+1}=\bigoplus_{\lambda_{i}>0}\left(
\mathcal{U}\times L(\lambda_{i})^{\infty}\right)  \oplus\mathcal{V}_{c}%
\oplus\bigoplus_{\lambda_{i}<0}\left(  \mathcal{U}\times L(\lambda
_{i})^{\infty}\right)  , \label{Selg_lift}%
\end{equation}
where $\mathcal{V}_{c}$ is the central Selgrade bundle%
\begin{equation}
\mathcal{V}_{c}=\left\{  (u,-re(u,0)+y,r)\left\vert u\in\mathcal{U},y\in
L^{0},r\in\mathbb{R}\right.  \right\}  . \label{central}%
\end{equation}

(ii) The dimensions of the subbundles of $\mathcal{U}\times\mathbb{R}^{n+1}$
are $\dim(\mathcal{U}\times L(\lambda_{i})^{\infty})=\dim L(\lambda_{i})$ for
all $i$ and $\dim\mathcal{V}_{c}=1+\dim L^{0}$.

(iii) For all $\lambda_{i}\not =0$ and every $(u,x,0)\in\mathcal{U}\times
L(\lambda_{i})^{\infty}$ the Lyapunov exponent is $\lambda_{i}$ and for every
$(u,x,r)\in\mathcal{V}_{c}$ the Lyapunov exponent is $0$.
\end{theorem}

\begin{proof}
(i), (ii) By Colonius, Santana, and Viscovini \cite[Theorem 4.3]{ColSV24}),
the lifted flow $\Phi^{1}$ has a Selgrade decomposition of the form%
\begin{align*}
\mathcal{U}\times\mathbb{R}^{n+1}  &  =\left(  \mathcal{U}\times L(\lambda
_{1})^{\infty}\right)  \oplus\cdots\oplus\left(  \mathcal{U}\times
L(\lambda_{\ell^{+}})^{\infty}\right)  \oplus\mathcal{V}_{c}\oplus\\
&  \qquad\qquad\oplus\left(  \mathcal{U}\times L(\lambda_{\ell^{+}+\ell^{0}%
+1})^{\infty}\right)  \oplus\cdots\oplus\left(  \mathcal{U}\times
L(\lambda_{\ell})^{\infty}\right)
\end{align*}
with $\lambda_{\ell^{+}}<0$ and $\lambda_{\ell^{+}+\ell^{0}+1}>0$. The
assertion on the dimension of $\left(  \mathcal{U}\times L(\lambda
_{i})^{\infty}\right)  $ is obvious. By \cite[Theorem 5.3]{ColSV24} the
central Selgrade bundle is given by (\ref{central}) and its dimension is
$1+\dim L^{0}$. It follows that $\ell^{0}=0$ if $0$ is not a Lyapunov exponent
and $\ell^{0}=1$ otherwise. This proves assertions (i) and (ii).

(iii) The assertion for $\mathcal{U}\times L(\lambda_{i})^{\infty}$ is clear.
For $(u,-re(u,0)+y,r)\in\mathcal{V}_{c}$ we obtain%
\begin{align*}
&  \frac{1}{t}\log\left\Vert \varphi^{1}(t,(-re(u,0)+y,r),u\right\Vert \\
&  =\frac{1}{t}\log\left\Vert \left(  e^{At}(-re(u,0)+y)+r\int_{0}%
^{t}e^{A(t-\sigma)}Bu(\sigma)d\sigma,r\right)  \right\Vert .
\end{align*}
Since we can omit the last component the assertion follows by Lemma
\ref{Lemma_growth1}.
\end{proof}

The chain control sets in the projective Poincar\'{e} sphere are determined by
\cite[Corollary 4.4 and Corollary 5.4]{ColSV24} and Theorem \ref{thmprojMorse}%
. Recall from (\ref{h1}) that $h^{1}(x)=\mathbb{P}(x,1)$.

\begin{theorem}
(i) For the induced system on the projective Poincar\'{e} sphere
$\mathbb{P}^{n}$ there is a unique chain control set $_{\mathbb{P}}E_{c}$ with
$_{\mathbb{P}}E_{c}\cap\mathbb{P}^{n,1}\not =\varnothing$. It is given by
$h^{1}(E)$, where $E$ is the unique chain control set of (\ref{linear}).

(iii) The other chain control sets on $\mathbb{P}^{n}$ are contained in the
projective equator $\mathbb{P}^{n,0}$ and are given by $\mathbb{P}%
(L(\lambda_{i})^{\infty})$ for $\lambda_{i}\not =0$.
\end{theorem}

Any bilinear control system can be projected to the unit sphere. For system
(\ref{lift2}) on $\mathbb{R}^{n+1}$ abbreviate%
\[
A_{0}:=\left(
\begin{array}
[c]{cc}%
A & 0\\
0 & 0
\end{array}
\right)  \text{ and }A_{i}=\left(
\begin{array}
[c]{cc}%
0 & b_{i}\\
0 & 0
\end{array}
\right)  \text{ for }i=1,\ldots,m.
\]
Then the induced system on $\mathbb{S}^{n}$ is described by%
\begin{equation}
\dot{s}(t)=h_{0}(s(t))+\sum_{i=1}^{m}u_{i}(t)h_{i}(s(t)),u\in\mathcal{U},
\label{S^n}%
\end{equation}
where $h_{i}(s)=\left[  A_{i}-s^{\top}A_{i}s\cdot I_{n}\right]  s$ for
$i=0,1,\ldots,m$.

Denote by $\pi:\mathbb{R}_{0}^{n+1}=\mathbb{R}^{n+1}\setminus\{0\}\rightarrow
\mathbb{S}^{n},y\mapsto y/\left\Vert y\right\Vert $, the canonical projection.
Note that the cocycle $\varphi^{1}$ defined in (\ref{Fi_1}) satisfies%
\begin{equation}
\pi\varphi^{1}(t,\pi(x,r),u)=\frac{\varphi^{1}(t,x,r,u)}{\left\Vert
\varphi^{1}(t,x,r,u)\right\Vert }. \label{Fi_1a}%
\end{equation}
The relations between the chain control sets on projective space and on the
sphere are described by Colonius and Santana \cite[Theorem 8]{ColS24b}. This
yields the following.

\begin{theorem}
\label{Theorem_twoS}Consider the bilinear control system (\ref{lift2}) on
$\mathbb{R}^{n+1}$ and the chain control set $_{\mathbb{P}}E_{c}$ of the
induced control system on $\mathbb{P}^{n}$.

(i) The set $_{\mathbb{S}}E_{c}^{0}:=\{s\in\mathbb{S}^{n}\left\vert
\mathbb{P}s\in\,_{\mathbb{P}}E_{c}\right.  \}$ is the unique chain control set
in $\mathbb{S}^{n}$ which projects onto $_{\mathbb{P}}E_{c}$ if and only if
there is $s_{0}\in\mathbb{S}^{n}$ with $\mathbb{P}s_{0}\in\,_{\mathbb{P}}%
E_{c}$ and $-s_{0}\in\mathbf{R}^{c}(s_{0})$.

(ii) There are two chain control sets $_{\mathbb{S}}E_{c}^{1}=-\,_{\mathbb{S}%
}E_{c}^{2}$ projecting onto $_{\mathbb{P}}E_{c}$ with%
\[
_{\mathbb{S}}E_{c}^{1}\cup\,_{\mathbb{S}}E_{c}^{2}=\{s\in\mathbb{S}%
^{n}\left\vert \mathbb{P}s\in\,_{\mathbb{P}}E_{c}\right.  \},
\]
if and only if for all $s_{0}\in\mathbb{S}^{n}$ with $\mathbb{P}s_{0}%
\in\,_{\mathbb{P}}E_{c}$ it holds that $-s_{0}\not \in \mathbf{R}^{c}(s_{0})$.

(iii) Define for $j\in\{0,1,2\}$ and $u\in\mathcal{U}$%
\[
F^{j}(u):=\left\{  x\in\mathbb{R}^{n+1}\left\vert \mathbb{\pi}\varphi
(t,x,u)\in\,_{\mathbb{S}}E_{c}^{j}\text{ for all }t\in\mathbb{R}\right.
\right\}  .
\]
In case (i) it follows that $F^{0}(u)$ is a linear subspace. In case (ii) it
follows for $j=1,2$ that $F^{j}(u)$ is a convex cone.
\end{theorem}

Thus there are one or two chain control sets $_{\mathbb{S}}E_{c}^{j},j=0$ or
$j=1,2$, on the Poincar\'{e} sphere $\mathbb{S}^{n}$ projecting onto the
central chain control set $_{\mathbb{P}}E_{c}$ in $\mathbb{P}^{n}$. They are
not subsets of the equator $\mathbb{S}^{n,0}$ and can also be obtained by
\begin{equation}
\mathbb{S}^{n}\cap\left\{  (x,r)\in\mathbb{R}^{n+1}\left\vert \exists
u\in\mathcal{U}:(u,x,r)\in\mathcal{V}_{c}\right.  \right\}  . \label{S_c}%
\end{equation}
Furthermore, the maximal chain transitive sets $\mathbb{P}(L(\lambda
_{i})),\lambda_{i}\not =0$, of the flow $\mathbb{P}(\psi,0)$ yield the chain
transitive sets $\mathcal{U}\times\mathbb{P}(L(\lambda_{i})^{\infty})$ of the
flow $\mathbb{P}\Phi^{1}$ restricted to the projective equator $\mathbb{P}%
^{n,0}$. By \cite[Theorem 8]{ColS24b}, each of them gives one or two maximal
chain transitive sets $\mathcal{U}\times\,_{\mathbb{S}}L(\lambda_{i}%
)_{j}^{\infty}\subset\mathcal{U}\times\mathbb{S}^{n,0},j=0$ or $j=1,2$. Here
$_{\mathbb{S}}L(\lambda_{i})_{0}^{\infty}$ is the unique subset of
$\mathbb{S}^{n,0}$ that projects to $\mathbb{P}(L(\lambda_{i})^{\infty})$ or
else $_{\mathbb{S}}L(\lambda_{i})_{1}^{\infty}$ and $_{\mathbb{S}}%
L(\lambda_{i})_{2}^{\infty}$ are two subsets of $\mathbb{S}^{n,0}$ that
project to $\mathbb{P}(L(\lambda_{j})^{\infty})$. If, for example,
$\lambda_{i}$ is a simple real eigenvalue, the first case occurs, if
$\lambda_{i}$ is the real part of a complex conjugate pair of eigenvalues, the
second case occurs.

We note the following corollary which clarifies the limit behavior of
trajectories on the Poincar\'{e} sphere. The $\alpha$- and $\omega$-limit sets
for a point $s_{0}\in\mathbb{S}^{n}$ and a control $u\in\mathcal{U}$ are%
\begin{align}
\alpha(s_{0},u)  &  :=\left\{  s\in\mathbb{S}^{n}\left\vert \exists
t_{k}\rightarrow-\infty:\pi\varphi^{1}(t_{k},s_{0},u)\rightarrow s\text{ for
}k\rightarrow\infty\right.  \right\}  ,\label{limit_control}\\
\omega(s_{0},u)  &  :=\left\{  s\in\mathbb{S}^{n}\left\vert \exists
t_{k}\rightarrow\infty:\pi\varphi^{1}(t_{k},s_{0},u)\rightarrow s\text{ for
}k\rightarrow\infty\right.  \right\}  .\nonumber
\end{align}
Recall that the homeomorphism $h^{1}:\mathbb{R}^{n}\rightarrow\mathbb{S}%
^{n,+}$ is defined in (\ref{h1}).

\begin{corollary}
\label{Corollary_limit}Consider the control system induced on the Poincar\'{e}
sphere $\mathbb{S}^{n}$ by a linear control system of the form (\ref{linear}).
For all $u\in\mathcal{U}$ and $s_{0}=h^{1}(x_{0})\in\mathbb{S}^{n,+}$ for
$x_{0}\in\mathbb{R}^{n}$, the limit sets $\alpha(s_{0},u)$ and $\omega
(s_{0},u)$ of the trajectory $\pi\varphi^{1}(t,s_{0},u)=h^{1}(\varphi
(t,x_{0},u))$ are contained in one of the central chain control sets
$_{\mathbb{S}}E_{c}^{j},j=0$ or $j=1,2$, or in one of the sets $_{\mathbb{S}%
}L(\lambda_{i})_{j}^{\infty}\subset\mathbb{S}^{n,0},j=0$ or $j=1,2,\lambda
_{i}\not =0$.
\end{corollary}

\begin{proof}
For any flow on a compact metric space, it is well known that the $\alpha$-
and $\omega$-limit sets are contained in the chain recurrent set; cf. e.g.
Alongi and Nelson \cite[Corollary 2.7.15]{AlonN07} or Colonius and Kliemann
\cite[Proposition 3.1.12]{ColK14} \ For the control flow associated with the
control system on $\mathbb{P}^{n}$ the chain recurrent set is the union of the
sets $\mathcal{U}\times\mathbb{P}L(\lambda_{i})^{\infty},\lambda_{i}\not =0$,
with the projected central subbundle $\mathbb{P}\mathcal{V}_{c}$, which
coincides with the lift of the central chain control set $_{\mathbb{P}}E_{c}$
to the chain transitive set $\left\{  (u,p)\in\mathcal{U}\times\mathbb{P}%
^{n}\left\vert \mathbb{P}\varphi^{1}(t,p,u)\in\,_{\mathbb{P}}E_{c}\text{ for
}t\in\mathbb{R}\right.  \right\}  $. Thus, for the control flow associated
with the control system on $\mathbb{S}^{n}$, the chain recurrent set is the
union of the sets $\mathcal{U}\times\,_{\mathbb{S}}L(\lambda_{i})_{j}^{\infty
},j=0$ or $j=1,2,\lambda_{i}\not =0$, with the intersection of the central
subbundle $\mathcal{V}_{c}$ and $\mathcal{U}\times\mathbb{S}^{n}$; cf.
(\ref{S_c}). Since the limit sets (\ref{limit_control}) are the projections of
the limit sets of the control flow the assertion follows.
\end{proof}

Corollary \ref{Corollary_limit} shows that the possible limit sets of
controlled trajectories are the central chain control sets $_{\mathbb{S}}%
E_{c}^{j}$ and the chain control sets $_{\mathbb{S}}L(\lambda_{i})_{j}%
^{\infty}$ on the equator. On the level of control flows, the possible limit
sets are the lifts of these chain control sets to $\mathcal{U\times}%
\mathbb{S}^{n}$ given by
\begin{align*}
_{\mathbb{S}}\mathcal{E}_{c}^{j}  &  :=\left\{  \left(  u,s\right)
\in\mathcal{U\times}\mathbb{S}^{n}\left\vert \pi\varphi^{1}(t,s,u)\in
\,_{\mathbb{S}}E_{c}^{j}\text{ for all }t\in\mathbb{R}\right.  \right\}  ,\\
\mathcal{U}\times\, _{\mathbb{S}}L(\lambda_{i})_{j}^{\infty}  &  :=\left\{  \left(  u,s\right)
\in\mathcal{U\times}\mathbb{S}^{n}\left\vert \pi\varphi^{1}(t,s,u)\in
\,_{\mathbb{S}}L(\lambda_{i})_{j}^{\infty}\text{ for all }t\in\mathbb{R}%
\right.  \right\}  .
\end{align*}
In the next section, we will determine the linearization about these sets.

\section{Selgrade decomposition on the sphere\label{Section4}}

In this section we determine the Selgrade decomposition and the corresponding
exponential growth rates for the linearization of the control flow induced by
a linear control system on the Poincar\'{e} sphere. Here we have to be more
careful concerning the tangent spaces. Some arguments are taken from Crauel
\cite[Section 3]{Crau91}, who considers random differential systems.

It is convenient to endow $\mathbb{R}^{n}$ with a scalar product, which makes
the Lyapunov spaces pairwise orthogonal, i.e., $\left\langle x,y\right\rangle
=0$ for $x\in L(\lambda_{i}),y\in L(\lambda_{j})$ with $\lambda_{i}%
\not =\lambda_{j}$. On $\mathbb{R}^{n+1}$ we use the scalar product%
\[
\left\langle (x,x_{n+1}),(y,y_{n+1})\right\rangle ^{\prime}=\left\langle
x,y\right\rangle +x_{n+1}\cdot y_{n+1},\text{ for }x,y\in\mathbb{R}^{n}\text{
and }x_{n+1},y_{n+1}\in\mathbb{R}.
\]
Note that the Lyapunov exponents are independent of the scalar product. The
tangent bundle $T\mathbb{R}_{0}^{n+1}~$of $\mathbb{R}_{0}^{n+1}$ is trivial
and is identified with $\mathbb{R}_{0}^{n+1}\times\mathbb{R}^{n+1}$.
Similarly, we identify the projective bundle $P\mathbb{R}_{0}^{n+1}$ with
$\mathbb{R}_{0}^{n+1}\times\mathbb{P}^{n}$. We consider the sphere
$\mathbb{S}^{n}$ as an embedded compact $n$-dimensional submanifold of
$\mathbb{R}_{0}^{n+1}$ and identify the tangent bundle $T\mathbb{S}^{n}$ with
the subset of $T\mathbb{R}_{0}^{n+1}$ given by%
\begin{equation}
T\mathbb{S}^{n}=\left\{  (s,v)\in\mathbb{S}^{n}\times\mathbb{R}^{n+1}%
\left\vert s\in\mathbb{S}^{n}\text{ and }\left\langle v,s\right\rangle
=0\right.  \right\}  . \label{HC_10}%
\end{equation}
Fix the Riemannian metric on $\mathbb{S}^{n}$ induced by this identification.
Note that $T\mathbb{S}^{n}$ can be written as a set of pairs $(s,v_{s})$ where
$s\in\mathbb{S}^{n}$ and $v_{s}\in T_{s}\mathbb{S}^{n}$. With $\pi
:\mathbb{R}_{0}^{n+1}\rightarrow\mathbb{S}^{n},\pi(x,x_{n+1}):=\frac
{(x,x_{n+1})}{\left\Vert (x,x_{n+1})\right\Vert }$, the points $s\in
\mathbb{S}^{n}$ may be written as $s=\pi(x,x_{n+1}),\left\Vert (x,x_{n+1}%
)\right\Vert =1$.

The induced control system on the Poincar\'{e} sphere $\mathbb{S}^{n}$
generates the control flow $\pi\Phi^{1}:\mathbb{R}\times\mathcal{U}%
\times\mathbb{S}^{n}\rightarrow\mathcal{U}\times\mathbb{S}^{n}$ given by%
\begin{equation}
\pi\Phi_{t}^{1}(u,s)=(u(t+\cdot),\pi\varphi^{1}(t,s,u)),t\in\mathbb{R}%
,s=\pi(x,x_{n+1}),u\in\mathcal{U}. \label{pi_phi}%
\end{equation}
This flow can be linearized with respect to the component in $\mathbb{S}^{n}$,
which yields the following continuous flow $T\pi\Phi^{1}:\mathbb{R}%
\times\mathcal{U}\times T\mathbb{S}^{n}\rightarrow\mathcal{U}\times
T\mathbb{S}^{n}$ given by%
\begin{equation}
T\left\vert _{s}\right.  \pi\Phi_{t}^{1}(u,s)(v,v_{n+1})=\left(
u(t+\cdot),\pi\varphi^{1}(t,s,u),D_{s}\pi\varphi^{1}(t,s,u)(v,v_{n+1})\right)
, \label{linearized0}%
\end{equation}
where $u\in\mathcal{U},s\in\mathbb{S}^{n}$, and $D_{s}\pi\varphi
^{1}(t,s,u)(v,v_{n+1})$ means the derivative of $\pi\varphi^{1}(t,s,u)$ with
respect to the second variable at the point $s$ applied to $(v,v_{n+1})\in
T_{s}\mathbb{S}^{n}$ as a linear map.

This is a linear flow on the vector bundle $\mathcal{U}\times T\mathbb{S}^{n}$
with base space $\mathcal{U}\times\mathbb{S}^{n}$. For a Lyapunov space
$L(\lambda_{i_{0}}),\lambda_{i_{0}}\not =0,$ of the matrix $A$, recall that
$L(\lambda_{i_{0}})^{\infty}=L(\lambda_{i_{0}})\times\{0\}\subset
\mathcal{U}\times\mathbb{R}^{n+1}$. By the remarks following Theorem
\ref{Theorem_twoS}, the set $L(\lambda_{i_{0}})^{\infty}\cap\mathbb{S}^{n}$
consists of one or two chain transitive sets denoted by $\mathcal{U}%
\times\,_{\mathbb{S}}L(\lambda_{i_{0}})_{j}^{\infty}\subset$ $\mathcal{U}%
\times\mathbb{S}^{n,0},j=0$ or $j=1,2$. We choose one of these sets and denote
it, for some notational simplification, by$\,_{\mathbb{S}}L(\lambda_{i_{0}%
})^{\infty}$. Define%
\begin{equation}
T_{\,_{\mathbb{S}}L(\lambda_{i_{0}})^{\infty}}\mathbb{S}^{n}:=\bigcup
\nolimits_{s\in_{\mathbb{S}}L(\lambda_{i_{0}})^{\infty}}T_{s}\mathbb{S}%
^{n}\text{ and }P_{\,_{\mathbb{S}}L(\lambda_{i_{0}})^{\infty}}\mathbb{S}%
^{n}:=\bigcup\nolimits_{s\in_{\mathbb{S}}L(\lambda_{i_{0}})^{\infty}%
}\mathbb{P}_{s}\mathbb{S}^{n}, \label{TandP}%
\end{equation}
where $\mathbb{P}_{s}\mathbb{S}^{n}:=\left\{  s\right\}  \times\left\{
\mathbb{P}(v,v_{n+1})\left\vert (v,v_{n+1})\in T_{s}\mathbb{S}^{n}\right.
\right\}  $.

\begin{proposition}
\label{Proposition_base}For the linearized flow (\ref{linearized0}), the base
space can be restricted to the compact invariant set $\mathcal{U}%
\times\,_{\mathbb{S}}L(\lambda_{i_{0}})^{\infty}$. This results in the
following flow defined on $\mathcal{U}\times T_{\,_{\mathbb{S}}L(\lambda
_{i_{0}})^{\infty}}\mathbb{S}^{n}$ given by%
\[
T\left\vert _{\pi(x,0)}\right.  \pi\Phi_{t}^{1}(u,s)(v,v_{n+1})=\left(
u(t+\cdot),\pi\varphi^{1}(t,x,0,u),D_{s}\pi\varphi^{1}(t,x,0,u)(v,v_{n+1)}%
)\right)  ,
\]
for $t\in\mathbb{R},u\in\mathcal{U},s=\pi(x,0)=(x,0)\in\,_{\mathbb{S}%
}L(\lambda_{i_{0}})^{\infty}$, and $(v,v_{n+1})\in T_{s}\mathbb{S}^{n}$. This
is a linear flow again denoted by $T\pi\Phi^{1}$ on a vector bundle with chain
transitive base space $\mathcal{U}\times\,_{\mathbb{S}}L(\lambda_{i_{0}%
})^{\infty}$ and Selgrade decomposition in the form%
\begin{equation}
\mathcal{U}\times T_{\,_{\mathbb{S}}L(\lambda_{i_{0}})^{\infty}}\mathbb{S}%
^{n}=\,_{\mathbb{S}}\mathcal{V}_{1}\oplus\cdots\oplus\,_{\mathbb{S}%
}\mathcal{V}_{k}. \label{Selg_sphere1}%
\end{equation}
Here the $_{\mathbb{S}}\mathcal{V}_{i}$ are invariant subbundles and their
projections to $P_{\,_{\mathbb{S}}L(\lambda_{i_{0}})^{\infty}}\mathbb{S}^{n}$
are the maximal chain transitive sets of the flow induced by $T\pi\Phi^{1}$.
\end{proposition}

\begin{proof}
The flows on the compact metric spaces $\mathcal{U}$ and $\,_{\mathbb{S}%
}L(\lambda_{i_{0}})^{\infty}$ are chain transitive. By Alongi and Nelson
\cite[Theorem 2.7.18]{AlonN07}, for both flows, it suffices to consider chains
with all jump times equal to $1$. This implies that the product flow on
$\mathcal{U}\times\,_{\mathbb{S}}L(\lambda_{i_{0}})^{\infty}$ is chain
transitive noting that the flow on $_{\mathbb{S}}L(\lambda_{i_{0}})^{\infty}$
does not depend on the element in $\mathcal{U}$. (Alternatively, this also
holds by the remarks following Theorem \ref{Theorem_twoS}.) By Selgrade's
theorem (cf. e.g. Colonius and Kliemann \cite[Theorem 9.2.5]{ColK14}), every
linear flow with chain transitive base space admits a Selgrade decomposition.
\end{proof}

In the following, we will determine the Selgrade bundles in
(\ref{Selg_sphere1}).

Since $\pi$ is a submersion, $T\pi:T\mathbb{R}_{0}^{n+1}\rightarrow
T\mathbb{S}^{n}$ acts as a projection along $\ker(T\pi)$. Note that
$\ker(T_{y}\pi)=\mathrm{span}(y)$. By (\ref{HC_10}) and, with the particular
choice of the Riemannian metric, $T\pi$ becomes an orthogonal projection, and
hence, for $y\in\mathbb{R}_{0}^{n+1}$ and $v\in T_{y}\mathbb{R}_{0}^{n+1}$,%
\begin{equation}
(T_{y}\pi)(y,v)=\left(  \pi y,\left\Vert y\right\Vert ^{-1}(v-\left\Vert
y\right\Vert ^{-2}\left\langle v,y\right\rangle y)\right)  \in T_{\pi
y}\mathbb{S}^{n}. \label{HC_11}%
\end{equation}
For $\left\Vert y\right\Vert =1$, this simplifies to $(T_{y}\pi
)(y,v)=(y,v-\left\langle v,y\right\rangle y)$. The geometric interpretation of
formula (\ref{HC_11}) is the following: Project $y\in\mathbb{R}_{0}^{n+1}$ to
$s=\pi y=\frac{y}{\left\Vert y\right\Vert }\in\mathbb{S}^{n}$. Then the
tangent vector $v\in T_{y}\mathbb{R}_{0}^{n+1}$ is mapped to $\left\Vert
y\right\Vert ^{-1}v$ minus the projection $\left\langle v,\frac{y}{\left\Vert
y\right\Vert }\right\rangle \frac{y}{\left\Vert y\right\Vert }$ to the radial component.

The spaces $L(\lambda_{i})^{\infty}=L(\lambda_{i})\times\{0\}$ are subspaces
of $\mathbb{R}^{n+1}$ and $\left\{  y\right\}  \times L(\lambda_{i})^{\infty}$
are subspaces of the tangent spaces $T_{y}\mathbb{R}_{0}^{n+1}\cong
\{y\}\times\mathbb{R}^{n+1}$. Consider the map$\,(\mathrm{id}_{\mathcal{U}%
},T_{(x,0)}\pi):\mathcal{U}\times T_{(x,0)}\mathbb{R}^{n+1}\rightarrow
\mathcal{U}\times T_{s}\mathbb{S}^{n}$.

\begin{lemma}
\label{Lemma_T}Let $s=\pi(x,0)\in\,_{\mathbb{S}}L(\lambda_{i_{0}})^{\infty
}\allowbreak\subset\mathbb{S}^{n,0}$ with $(x,0)\in\mathbb{R}^{n+1}$. Then one
obtains
\begin{align*}
\,(\mathrm{id}_{\mathcal{U}},T_{(x,0)}\pi)(\mathcal{U}\times\left\{
(x,0)\right\}  \times L(\lambda_{i})^{\infty})  &  =\mathcal{U}\times\left\{
(s,v,0)\left\vert (v,0)\in L(\lambda_{i})^{\infty}\right.  \right\}
,\lambda_{i}\not =\lambda_{i_{0}},\\
\,(\mathrm{id}_{\mathcal{U}},T_{(x,0)}\pi)(\mathcal{U}\times\left\{
(x,0)\right\}  \times L(\lambda_{i_{0}})^{\infty})  &  =\mathcal{U}%
\times\left\{  (s,v-\left\langle v,s\right\rangle s,0)\left\vert (v,0)\in
L(\lambda_{i_{0}})^{\infty}\right.  \right\}  ,\\
(\mathrm{id}_{\mathcal{U}},T_{(x,0)}\pi)(\left\{  (x,0)\right\}
\times\mathcal{V}_{c})  &  =\left\{  \left(  u,s,w-\left\langle
w,s\right\rangle s,r\right)  \left\vert (u,w,r)\in\mathcal{V}_{c}\right.
\right\}  .
\end{align*}
The kernel of $T_{(x,0)}\pi$ restricted to $\left\{  (x,0)\right\}  \times
L(\lambda_{i_{0}})$ equals the span of $(x,0)$.
\end{lemma}

\begin{proof}
For $(x,0)\in\mathbb{S}^{n}$ and $(v,v_{n+1})\in\mathbb{R}^{n}\times
\mathbb{R}\cong T_{(x,0)}\mathbb{R}_{0}^{n+1}$, formula (\ref{HC_11}) implies%
\begin{align*}
(T_{(x,0)}\pi)(x,0)(v,v_{n+1})  &  =\left(  \pi(x,0),(v,v_{n+1})-\left\Vert
x\right\Vert ^{-2}\left\langle (v,v_{n+1}),(x,0)\right\rangle (x,0)\right) \\
&  =\left(  s,v-\left\langle v,\frac{x}{\left\Vert x\right\Vert }\right\rangle
\frac{x}{\left\Vert x\right\Vert },v_{n+1}\right)  .
\end{align*}
For $i\not =i_{0}$ and $s=\pi(x,0)\in\,_{\mathbb{S}}L(\lambda_{i_{0}}%
)^{\infty}\allowbreak\subset\mathbb{S}^{n}$ orthogonality of $L(\lambda
_{i_{0}})$ and $L(\lambda_{i})$ implies that
\begin{align*}
(T_{(x,0)}\pi)\left(  (x,0),L(\lambda_{i})^{\infty}\right)   &  =\left\{
\left(  (s,v-\left\langle v,\frac{x}{\left\Vert x\right\Vert }\right\rangle
\frac{x}{\left\Vert x\right\Vert },0\right)  \left\vert (v,0)\in L(\lambda
_{i})^{\infty}\right.  \right\} \\
&  =\left\{  (s,v,0)\left\vert (v,0)\in L(\lambda_{i})^{\infty}\right.
\right\}  \subset T_{s}\mathbb{S}^{n}.
\end{align*}
In particular, $\left(  T_{(x,0)}\pi\right)  :T_{(x,0)}L(\lambda
_{i})\rightarrow T_{s}\mathbb{S}^{n}$ is injective. For $i=i_{0}$, the kernel
of $T_{(x,0)}\pi$ restricted to $L(\lambda_{i_{0}})^{\infty}$ equals the span
of $(x,0)$: In fact, by (\ref{HC_11}), the equality $T_{(x,0)}\pi\left(
(x,0)(v,v_{n+1})\right)  =(x,0,(0,0))$ holds if and only if%
\begin{align*}
(0,0)  &  =(v,v_{n+1})-\left\Vert x\right\Vert ^{-2}\left\langle
(v,v_{n+1}),(x,0)\right\rangle (x,0)=(v,v_{n+1})-\left\Vert x\right\Vert
^{-2}\left\langle v,x\right\rangle (x,0)\\
&  =\left(  v-\left\langle v,\frac{x}{\left\Vert x\right\Vert }\right\rangle
\frac{x}{\left\Vert x\right\Vert },v_{n+1}\right)  .
\end{align*}
This holds if and only if $v_{n+1}=0$ and $(v,v_{n+1})$ is in the span of
$(x,0)$.

Recall that $\mathcal{V}_{c}$ is given by (\ref{central}) and compute using
(\ref{HC_11})%
\begin{align*}
&  (T_{(x,0)}\pi)\left(  (x,0),(-re(u,0)+y,r)\right) \\
&  =\left(  s,-re(u,0)+y,r)-\left\langle (-re(u,0)+y,r),\pi(x,0)\right\rangle
\pi(x,0)\right) \\
&  =\left(  s,-re(u,0)+y-\left\langle -re(u,0)+y,s\right\rangle s,r\right)  .
\end{align*}
With $w:=-re(u,0)+y$ it follows that%
\[
(\mathrm{id}_{\mathcal{U}},T_{(x,0)}\pi)(\left\{  (x,0)\right\}
\times\mathcal{V}_{c})=\left\{  \left(  u,s,w-\left\langle w,s\right\rangle
s,r\right)  \left\vert (u,w,r)\in\mathcal{V}_{c}\right.  \right\}  .
\]

\end{proof}

Next we determine the derivative of the cocycle $\pi\varphi^{1}(t,x,0,u)$ in a
point $s_{0}\in\mathbb{S}^{n,0}$ in direction $(v,v_{n+1})$.

\begin{lemma}
\label{Lemma_derivative}The derivative of the map $\pi\varphi^{1}%
(t,\cdot,u):\mathbb{S}^{n}\rightarrow\mathbb{S}^{n}$ in the point
$s=\pi(x,0)\in\mathbb{S}^{n}$ in direction $(v,v_{n+1})\in T_{s}\mathbb{S}%
^{n}$ is the element of $T_{\pi\varphi^{1}(t,s,u)}\mathbb{S}^{n}$ given by%
\begin{align}
&  D_{s}\pi\varphi^{1}(t,x,0,u)(v,v_{n+1)})\nonumber\\
&  =\left\Vert e^{At}x\right\Vert ^{-1}\left(  e^{At}v+v_{n+1}\int_{0}%
^{t}e^{A(t-\sigma)}Bu(\sigma)d\sigma,v_{n+1}\right) \label{D_phi}\\
&  \qquad-\left\Vert e^{At}x\right\Vert ^{-3}\left(  e^{At}x,0\right)
\left\langle e^{At}v+v_{n+1}\int_{0}^{t}e^{A(t-\sigma)}Bu(\sigma
)d\sigma,e^{At}x\right\rangle .\nonumber
\end{align}

\end{lemma}

\begin{proof}
By (\ref{Fi_1a}), the cocycle on $\mathbb{S}^{n}$ satisfies%
\[
\pi\varphi^{1}(t,\pi(x,0),u)=\frac{\varphi^{1}(t,x,0,u)}{\left\langle
\varphi^{1}(t,x,0,u),\varphi^{1}(t,x,0,u)\right\rangle ^{1/2}}.
\]
We compute for the derivatives in $(x,0)$ in direction $(v,v_{n+1})\in
T_{s}\mathbb{S}^{n}$%
\[
D_{s}\varphi^{1}(t,x,0,u)(v,v_{n+1)})=\left(  e^{At}v+v_{n+1}\int_{0}%
^{t}e^{A(t-\sigma)}Bu(\sigma)d\sigma,v_{n+1}\right)
\]
and%
\begin{align*}
&  D_{s}\left\langle \varphi^{1}(t,x,0,u),\varphi^{1}(t,x,0,u)\right\rangle
^{1/2}(v,v_{n+1})\\
&  =\frac{\left\langle (e^{At}v+v_{n+1}\int_{0}^{t}e^{A(t-\sigma)}%
Bu(\sigma)d\sigma,v_{n+1}),\varphi^{1}(t,x,0,u)\right\rangle }{\left\langle
\varphi^{1}(t,x,0,u),\varphi^{1}(t,x,0,u)\right\rangle ^{1/2}}\\
&  =\left\Vert e^{At}x\right\Vert ^{-1}\left\langle e^{At}v+v_{n+1}\int
_{0}^{t}e^{A(t-\sigma)}Bu(\sigma)d\sigma,e^{At}x\right\rangle .
\end{align*}
This yields%
\begin{align*}
&  D_{s}\pi\varphi^{1}(t,x,0,u)(v,v_{n+1})\\
&  =\left\Vert \varphi^{1}(t,x,0,u)\right\Vert ^{-2}\left[  D_{s}\varphi
^{1}(t,x,0,u)(v,v_{n+1})\cdot\left\langle \varphi^{1}(t,x,0,u),\varphi
^{1}(t,x,0,u)\right\rangle ^{1/2}\right. \\
&  \qquad\qquad-\left.  \varphi^{1}(t,x,0,u)\cdot D_{s}\left\langle
\varphi^{1}(t,x,0,u),\varphi^{1}(t,x,0,u)\right\rangle ^{1/2}(v,v_{n+1}%
)\right] \\
&  =\left\Vert e^{At}x\right\Vert ^{-2}\left[  \left(  e^{At}v+v_{n+1}\int
_{0}^{t}e^{A(t-\sigma)}Bu(\sigma)d\sigma,v_{n+1})\right)  \right.  \left\Vert
\left(  e^{At}x,0\right)  \right\Vert \\
&  \qquad\qquad-\left.  \left(  e^{At}x,0\right)  \left\Vert e^{At}%
x\right\Vert ^{-1}\left\langle e^{At}v+v_{n+1}\int_{0}^{t}e^{A(t-\sigma
)}Bu(\sigma)d\sigma,e^{At}x\right\rangle \right] \\
&  =\left\Vert e^{At}x\right\Vert ^{-1}\left(  e^{At}v+v_{n+1}\int_{0}%
^{t}e^{A(t-\sigma)}Bu(\sigma)d\sigma,v_{n+1}\right) \\
&  \qquad-\left\Vert e^{At}x\right\Vert ^{-3}\left(  e^{At}x,0\right)
\left\langle e^{At}v+v_{n+1}\int_{0}^{t}e^{A(t-\sigma)}Bu(\sigma
)d\sigma,e^{At}x\right\rangle .
\end{align*}
\bigskip
\end{proof}

For the control flow $\pi\Phi^{1}$ defined in (\ref{pi_phi}) with
linearization $T\pi\Phi$ defined in (\ref{linearized0}), the Lyapunov exponent
of a point $s=(x,x_{n+1})\in\mathbb{S}^{n}\subset\mathbb{R}^{n+1}$ and control
$u\in\mathcal{U}$ in direction $(0,0)\not =(v,v_{n+1})\in T_{s}\mathbb{S}%
^{n}\subset\mathbb{R}^{n}\times\mathbb{R}$ is given by%
\[
\lambda(s,u;v,v_{n+1}):=\lim_{\left\vert t\right\vert \rightarrow\infty}%
\frac{1}{t}\log\left\Vert D_{(x,x_{n+1})}\pi\varphi^{1}(t,x,x_{n+1}%
,u)(v,v_{n+1)})\right\Vert ,
\]
where $(u,s,v,v_{n+1})\in\mathcal{U}\times T_{s}\mathbb{S}^{n}$. The following
theorem describes the Selgrade bundles and their Lyapunov exponents.

\begin{theorem}
\label{Theorem_LinLyap}Consider the projected linear control flow $\pi\Phi
^{1}$ on $\mathcal{U}\times\mathbb{S}^{n}$ associated with the lift
(\ref{lift2}) of a linear control system of the form (\ref{linear}) and the
linearized flow $T\pi\Phi$ with base space restricted to $\mathcal{U}%
\times\,_{\mathbb{S}}L(\lambda_{i_{0}})^{\infty}$, as described in Proposition
\ref{Proposition_base}.

(i) Then the Selgrade bundles in (\ref{Selg_sphere1}) have the following form:%
\begin{align*}
_{\mathbb{S}}\mathcal{V}_{c}  &  =(\mathrm{id}_{\mathcal{U}},T\pi
)\mathcal{V}_{c}\text{ with }\dim\,_{\mathbb{S}}\mathcal{V}_{c}=\dim
L^{0}+1;\\
_{\mathbb{S}}\mathcal{V}_{i}  &  =(\mathrm{id}_{\mathcal{U}},T\pi)\left(
\mathcal{U}\times L(\lambda_{i})^{\infty}\right)  =\mathcal{U}\times
\,_{\mathbb{S}}L(\lambda_{i_{0}})^{\infty}\times L(\lambda_{i})^{\infty}
\end{align*}
with $\dim\,_{\mathbb{S}}\mathcal{V}_{i}=\dim L(\lambda_{i})$,
for $\lambda_{i}\not =0,\lambda_{i_{0}}$. If $\dim L(\lambda_{i_{0}})>1$ there
is an additional Selgrade bundle given by%
\[
_{\mathbb{S}}\mathcal{V}_{i_{0}}=(\mathrm{id}_{\mathcal{U}},T\pi
)(\mathcal{U}\times L(\lambda_{i_{0}})^{\infty})\text{ with }\dim
\,_{\mathbb{S}}\mathcal{V}_{i_{0}}=\dim L(\lambda_{i_{0}})-1.
\]

(ii) The Lyapunov exponents are%
\begin{align*}
\lambda(s,u;v,0)  &  =\lambda_{i}-\lambda_{i_{0}}\text{ for all }%
(u,s,v,0)\in\,_{\mathbb{S}}\mathcal{V}_{i}\text{ with }\lambda_{i}%
\not =0,\lambda_{i_{0}},\\
\lambda(s,u;v,v_{n+1})  &  =-\lambda_{i_{0}}\text{ for \thinspace
}(u,s,v,v_{n+1})\in\,_{\mathbb{S}}\mathcal{V}_{c}.
\end{align*}
For $\dim L(\lambda_{i_{0}})>1$%
\begin{equation}
\lambda(s,u;v,0)=0\text{ for all }(u,s,v,0)\in\,_{\mathbb{S}}\mathcal{V}%
_{i_{0}}. \label{lambda_i0}%
\end{equation}

\end{theorem}

\begin{proof}
(i) Recall that $T_{s}\mathbb{S}^{n}=\left\{  s\right\}  \times\left\{
v\in\mathbb{R}^{n+1}\left\vert \left\langle v,s\right\rangle =0\right.
\right\}  $ and $T\left(  \mathbb{R}_{0}^{n}\times\mathbb{R}\right)  \subset
T\mathbb{R}_{0}^{n+1}$. The Selgrade decomposition for the linear control flow
$\Phi^{1}$ associated with the lifted control system (\ref{lift2}) is given in
(\ref{Selg_lift}). Let $s=\pi(x,0)\in\,_{\mathbb{S}}L(\lambda_{i_{0}}%
)^{\infty}$. By Lemma \ref{Lemma_T}, the surjection
\[
(\mathrm{id}_{\mathcal{U}},T\pi):\mathcal{U}\times T_{\,_{\mathbb{S}}%
L(\lambda_{i_{0}})^{\infty}}\mathbb{R}_{0}^{n+1}\rightarrow\mathcal{U}\times
T_{_{\mathbb{S}}L(\lambda_{i_{0}})^{\infty}}\mathbb{S}^{n},
\]
maps $\mathcal{U}\times L(\lambda_{i_{0}})^{\infty}\times L(\lambda
_{i})^{\infty}\subset\mathcal{U}\times T_{L(\lambda_{i_{0}})^{\infty}%
}\mathbb{R}_{0}^{n+1},i\not =0,i_{0}$, onto $\mathcal{U}\times\,_{\mathbb{S}%
}L(\lambda_{i_{0}})^{\infty}\times L(\lambda_{i})^{\infty}$. Furthermore,
$\mathcal{U}\times\,_{\mathbb{S}}L(\lambda_{i_{0}})^{\infty}\times
\mathcal{V}_{c}$ is mapped onto%
\[
_{\mathbb{S}}\mathcal{V}_{c}=\left\{  \left(  u,s,w-\left\langle w,\frac
{x}{\left\Vert x\right\Vert }\right\rangle \frac{x}{\left\Vert x\right\Vert
},r\right)  \left\vert s\in\,_{\mathbb{S}}L(\lambda_{i_{0}})^{\infty}\text{
and }(u,w,r)\in\mathcal{V}_{c}\right.  \right\}  .
\]
Finally, $\mathcal{U}\times L(\lambda_{i_{0}})^{\infty}\times L(\lambda
_{i_{0}})^{\infty}$ is mapped onto%
\begin{align*}
&  (\mathrm{id}_{\mathcal{U}},T\pi)\left(  \mathcal{U}\times L(\lambda_{i_{0}%
})^{\infty}\times L(\lambda_{i_{0}})^{\infty}\right) \\
&  =\mathcal{U}\times\left\{  (s,v-\left\langle v,\frac{x}{\left\Vert
x\right\Vert }\right\rangle \frac{x}{\left\Vert x\right\Vert },0)\left\vert
s\in\,_{\mathbb{S}}L(\lambda_{i_{0}})^{\infty}\text{ and }(v,0)\in
L(\lambda_{i_{0}})^{\infty}\right.  \right\}  .
\end{align*}
Let
\[
P_{L(\lambda_{i_{0}})^{\infty}}\mathbb{R}_{0}^{n+1}:=\bigcup
\nolimits_{(x,0)\in L(\lambda_{i_{0}})^{\infty}}P_{(x,0)}\mathbb{R}_{0}%
^{n+1}=\bigcup\nolimits_{(x,0)\in L(\lambda_{i_{0}})^{\infty}}\left\{
(x,0)\right\}  \times\mathbb{P}^{n}.
\]
In the projective bundle $\mathcal{U}\times P_{L(\lambda_{i_{0}})^{\infty}%
}\mathbb{R}_{0}^{n+1}$, the images of the Selgrade bundles $\mathcal{V}_{c}$
and $\mathcal{U}\times L(\lambda_{i})^{\infty}$ are the maximal chain
transitive sets. Their images in $P_{\,_{\mathbb{S}}L(\lambda_{i_{0}}%
)^{\infty}}\mathbb{S}^{n}$ (cf. (\ref{TandP}) are also chain transitive since
they are the continuous images of chain transitive sets in a compact metric space.

The same arguments show that the image of the Selgrade $\mathcal{U}\times$
$_{\mathbb{S}}L(\lambda_{i_{0}})^{\infty}\times L(\lambda_{i_{0}})^{\infty}$
is chain transitive. It follows that the subbundles defined in assertion (i)
are the Selgrade bundles. The kernel of $T_{(x,0)}\pi$ restricted to
$L(\lambda_{i_{0}})^{\infty}$ equals the span of $(x,0)$. Hence%
\[
\dim\,_{\mathbb{S}}\mathcal{V}_{i_{0}}=\dim\left(  T_{(x,0)}\pi\right)
L(\lambda_{i_{0}})^{\infty}=\dim L(\lambda_{i_{0}})^{\infty}-1=\dim
L(\lambda_{i_{0}})-1.
\]
The dimensions of the subbundles$\,_{\mathbb{S}}\mathcal{V}_{c}$ and
$_{\mathbb{S}}\mathcal{V}_{i},i\not =i_{0}$ are preserved under $(\mathrm{id}%
_{\mathcal{U}},T\pi)$ and%
\begin{align*}
\sum_{\lambda_{i}\not =0}\dim\,_{\mathbb{S}}\mathcal{V}_{i}+\dim
\,_{\mathbb{S}}\mathcal{V}_{c}  &  =\sum_{\lambda_{i}\not =0,\lambda_{i_{0}}%
}\dim L(\lambda_{i})+\dim L(\lambda_{i_{0}})-1+\dim L^{0}+1\\
&  =n=\dim T_{s}\mathbb{S}^{n}.
\end{align*}
(ii) By Lemma \ref{Lemma_T}, any direction in $(T_{(x,0)}\pi)L(\lambda
_{i})^{\infty}$ satisfies $(v,v_{n+1})=(v,0)\in L(\lambda_{i})^{\infty}$. Thus
Lemma \ref{Lemma_derivative} implies that%
\begin{align}
&  D_{s}\pi\varphi^{1}(t,x,0,u)(v,0)\label{4.9}\\
&  =\left\Vert e^{At}x\right\Vert ^{-1}\left(  e^{At}v,0\right)  -\left\Vert
e^{At}x\right\Vert ^{-3}\left(  e^{At}x,0\right)  \left\langle e^{At}%
v,e^{At}x\right\rangle .\nonumber
\end{align}
For $i\not =i_{0}$, the $A$-invariant subspaces $L(\lambda_{i})$ and
$L(\lambda_{i_{0}})$ are orthogonal, and hence $\left\langle e^{At}%
v,e^{At}x\right\rangle \allowbreak=0$ implying%
\[
D_{s}\pi\varphi^{1}(t,x,0,u)(v,0)=\left\Vert e^{At}x\right\Vert ^{-1}\left(
e^{At}v,0\right)  .
\]
For the Lyapunov exponent in direction $(v,0)\in L(\lambda_{i})^{\infty
},i\not =i_{0}$, this implies%
\begin{align*}
&  \lim_{t\rightarrow\pm\infty}\frac{1}{t}\log\left\Vert D_{s}\pi\varphi
^{1}(t,x,0,u)(v,0)\right\Vert =\lim_{t\rightarrow\pm\infty}\frac{1}{t}%
\log\frac{\left\Vert (e^{At}v,0)\right\Vert }{\left\Vert e^{At}x\right\Vert
}\\
&  =\lim_{t\rightarrow\pm\infty}\frac{1}{t}\log\left\Vert e^{At}v\right\Vert
-\lim_{t\rightarrow\pm\infty}\frac{1}{t}\log\left\Vert e^{At}x\right\Vert
=\lambda_{i}-\lambda_{i_{0}}.
\end{align*}
Next consider the Lyapunov exponents for $(u,s,v,0)\in(\mathrm{id}%
_{\mathcal{U}},T_{(x,0)}\pi)\mathcal{V}_{i_{0}}$. Formula (\ref{4.9}) yields%
\[
D_{s}\pi\varphi^{1}(t,x,0,u)(v,0)=\left\Vert e^{At}x\right\Vert ^{-1}\left[
\left(  e^{At}v,0\right)  -\left\Vert e^{At}x\right\Vert ^{-2}\left(
e^{At}x,0\right)  \left\langle e^{At}v,e^{At}x\right\rangle \right]  .
\]
The first factor has exponential growth $-\lambda_{i_{0}}$ for $t\rightarrow
\pm\infty$. By (\ref{HC_11}), the second factor is the projection of $\left(
e^{At}v,0\right)  \in T_{e^{At}x}\mathbb{R}_{0}^{n+1}$ to $\left(  T_{e^{At}%
x}\pi\right)  (e^{At}x)(v,0))$, which is also in $L(\lambda_{i_{0}})^{\infty}%
$.\textbf{ }Thus the exponential growth of the second factor coincides with
$\lambda_{i_{0}}$.\ It follows that%
\begin{align*}
&  \lim_{t\rightarrow\pm\infty}\frac{1}{t}\log\left\Vert D_{s}\pi\varphi
^{1}(t,x,0,u)(v,0)\right\Vert \\
&  =\lim_{t\rightarrow\pm\infty}\frac{1}{t}\log\left[  \left\Vert
e^{At}x\right\Vert ^{-1}\left\Vert \left(  e^{At}v,0\right)  -\left\Vert
e^{At}x\right\Vert ^{-2}\left(  e^{At}x,0\right)  \left\langle e^{At}%
v,e^{At}x\right\rangle \right\Vert \right] \\
&  =\lambda_{i_{0}}-\lambda_{i_{0}}=0.
\end{align*}
This proves (\ref{lambda_i0}).

Finally, we compute \, the \, Lyapunov \, exponents \, for \, $(u,s,v,v_{n+1})$ \,
$\in$ 
$(\mathrm{id}_{\mathcal{U}},T_{(x,0)}\pi)\mathcal{V}_{c}$. Lemma \ref{Lemma_T}
and equation (\ref{central}) imply
\[
(v,v_{n+1})=(w-\left\langle w,x\right\rangle x,r)\text{ for }%
(u,w,r)=(u,-re(u,0)+y,r)\in\mathcal{V}_{c}.
\]
Inserting this into formula (\ref{D_phi}), we obtain for $D_{s}\pi\varphi
^{1}(t,\pi(x,0),u)(v,v_{n+1})$ the following summands:%
\begin{align}
&  \left\Vert e^{At}x\right\Vert ^{-1}\left(  e^{At}(w-\left\langle
w,x\right\rangle x)+r\int_{0}^{t}e^{A(t-\sigma)}Bu(\sigma)d\sigma,r\right)
\nonumber\\
&  =\left\Vert e^{At}x\right\Vert ^{-1}\left(  e^{At}w+r\int_{0}%
^{t}e^{A(t-\sigma)}Bu(\sigma)d\sigma,r\right)  -\left\langle w,x\right\rangle
\left(  \frac{e^{At}x}{\left\Vert e^{At}x\right\Vert },0\right)  ,
\label{first}%
\end{align}
and%
\begin{align}
&  \left\Vert e^{At}x\right\Vert ^{-3}\left(  e^{At}x,0\right)  \left\langle
e^{At}(w-\left\langle w,x\right\rangle x)+r\int_{0}^{t}e^{A(t-\sigma
)}Bu(\sigma)d\sigma,e^{At}x\right\rangle \nonumber\\
&  =\left\Vert e^{At}x\right\Vert ^{-3}\left(  e^{At}x,0\right)  \left\langle
e^{At}w+r\int_{0}^{t}e^{A(t-\sigma)}Bu(\sigma)d\sigma,e^{At}x\right\rangle
\nonumber\\
&  \qquad-\left\Vert e^{At}x\right\Vert ^{-3}\left(  e^{At}x,0\right)
\left\langle e^{At}\left\langle w,x\right\rangle x,e^{At}x\right\rangle .
\label{second}%
\end{align}
By Lemma \ref{Lemma_growth1}, the term $e^{At}w+r\int_{0}^{t}e^{A(t-\sigma
)}Bu(\sigma)d\sigma$ has at most polynomial growth, and hence the first
summand in (\ref{first}) has exponential growth $-\lambda_{i_{0}}$.
Furthermore, the term $\left\langle w,x\right\rangle \frac{e^{At}x}{\left\Vert
e^{At}x\right\Vert }$ has constant norm. The first summand in (\ref{second})
equals%
\begin{equation}
\left\Vert e^{At}x\right\Vert ^{-1}\left(  \frac{e^{At}x}{\left\Vert
e^{At}x\right\Vert },0\right)  \left\langle e^{At}w+r\int_{0}^{t}%
e^{A(t-\sigma)}Bu(\sigma)d\sigma,\frac{e^{At}x}{\left\Vert e^{At}x\right\Vert
}\right\rangle . \label{third1}%
\end{equation}
By Lemma \ref{Lemma_growth1}, $e^{At}w+r\int_{0}^{t}e^{A(t-\sigma)}%
Bu(\sigma)d\sigma$ has at most polynomial growth and%
\[
\left\vert \left\langle e^{At}w+r\int_{0}^{t}e^{A(t-\sigma)}Bu(\sigma
)d\sigma,\frac{e^{At}x}{\left\Vert e^{At}x\right\Vert }\right\rangle
\right\vert \leq\left\Vert e^{At}w+r\int_{0}^{t}e^{A(t-\sigma)}Bu(\sigma
)d\sigma\right\Vert .
\]
It \, follows \, that \, the  \, exponential \, growth \, rate \, of \,(\ref{third1}) \, equals \,
$\lim_{t\rightarrow\pm\infty}\frac{1}{t}\log\left\Vert e^{At}x\right\Vert
^{-1}=-\lambda_{i_{0}}$. Finally, the second summand in (\ref{first}) is%
\[
\left\langle w,x\right\rangle \left(  \frac{e^{At}x}{\left\Vert e^{At}%
x\right\Vert },0\right)  \left\langle \frac{e^{At}x}{\left\Vert e^{At}%
x\right\Vert },\frac{e^{At}x}{\left\Vert e^{At}x\right\Vert }\right\rangle ,
\]
which has bounded norm.

We conclude that the term in (\ref{second}) has exponential growth
$-\lambda_{i_{0}}$. It follows that both summands (\ref{first}) and
(\ref{second}) of $D_{s}\pi\varphi^{1}(t,\pi(x,0),u)(v,v_{n+1})$ have
exponential growth rate $-\lambda_{i_{0}}$. In general, the exponential growth
of two summands with equal exponential growth rate $\lambda$ (given by the
limit superior) is equal to or less than $\lambda$. But here we actually have
limits for $t\rightarrow\pm\infty$ and also for the inverses. This implies
that the sum has exponential growth $\lambda$; cf. Cesari \cite[(3.12.v)]%
{Cesa71}. This concludes the proof.
\end{proof}

For $s\in\mathbb{S}^{n}$ and $u\in\mathcal{U}$ a direction is called stable,
if the corresponding Lyapunov exponents for $t\rightarrow\pm\infty$ are
negative and unstable if they are positive. An immediate corollary of Theorem
\ref{Theorem_LinLyap} is the following.

\begin{corollary}
\label{Corollary_stable_sub}In the situation of Theorem \ref{Theorem_LinLyap},
for $\lambda_{i}\not =0$, the subbundles $_{\mathbb{S}}\mathcal{V}_{i}%
,i>i_{0}$, and $_{\mathbb{S}}\mathcal{V}_{i},i<i_{0}$, consist for all
$u\in\mathcal{U}$ of stable and unstable directions, respectively. The
directions in $_{\mathbb{S}}\mathcal{V}_{c}$ $\ $are stable for $\lambda
_{i_{0}}>0$ and unstable for $\lambda_{i_{0}}<0$.

(i) Define%
\begin{align}
\text{for }\lambda_{i_{0}}  &  >0,\,_{\mathbb{S}}\mathcal{V}^{-}%
:=\bigoplus_{i>i_{0},\lambda_{i}\not =0}\,_{\mathbb{S}}\mathcal{V}_{i}%
\,\oplus\,_{\mathbb{S}}\mathcal{V}_{c}\text{ and }_{\mathbb{S}}\mathcal{V}%
^{+}:=\bigoplus_{i<i_{0}}\,_{\mathbb{S}}\mathcal{V}_{i},\label{V_inv1}\\
\text{for }\lambda_{i_{0}}  &  <0,\,_{\mathbb{S}}\mathcal{V}^{-}%
:=\bigoplus_{i>i_{0}}\,_{\mathbb{S}}\mathcal{V}_{i}\text{ and }_{\mathbb{S}%
}\mathcal{V}^{+}:=\,_{\mathbb{S}}\mathcal{V}_{c}\oplus\bigoplus_{i<i_{0}%
,\lambda_{i}\not =0}\,_{\mathbb{S}}\mathcal{V}_{i}. \label{V_inv2}%
\end{align}
Then%
\[
\mathcal{U}\times T_{_{\mathbb{S}}L(\lambda_{i_{0}})^{\infty}}\mathbb{S}%
^{n}=\,_{\mathbb{S}}\mathcal{V}^{-}\oplus\,_{\mathbb{S}}\mathcal{V}_{i_{0}%
}\oplus\,_{\mathbb{S}}\mathcal{V}^{+}%
\]
is a decomposition into the stable, center, and unstable subbundles.

(ii) For the stable subbundle $_{\mathbb{S}}\mathcal{V}^{-}$, the supremal
exponential growth rate $\kappa\left(  _{\mathbb{S}}\mathcal{V}^{-}\right)
:=\sup\{\left.  \lambda(u,x,v^{-})\right\vert \allowbreak(u,x,v^{-}%
)\allowbreak\in\,_{\mathbb{S}}\mathcal{V}^{-}\}$ satisfies%
\begin{align*}
\text{for }\lambda_{i_{0}}  &  >0,~\text{ }\kappa(_{\mathbb{S}}\mathcal{V}%
^{-})=\lambda_{i_{1}}-\lambda_{i_{0}}\text{ where }\lambda_{i_{1}}%
=\max\{\lambda_{i}\left\vert \lambda_{i}>0\text{ and }i\geq i_{0}\right.
\},\\
\text{for }\lambda_{i_{0}\text{ }}  &  <0,~\kappa\left(  _{\mathbb{S}%
}\mathcal{V}^{-}\right)  =\lambda_{\ell}-\lambda_{i_{0}}\text{ if }i_{0}%
<\ell\text{ and if }i_{0}=\ell\text{ then }_{\mathbb{S}}\mathcal{V}^{-}\text{
is trivial}.
\end{align*}

\end{corollary}

\begin{proof}
Recall that $\lambda_{1}>\cdots>\lambda_{\ell}$. Theorem \ref{Theorem_LinLyap}%
(iii) shows that%
\begin{equation}
\kappa(\,_{\mathbb{S}}\mathcal{V}_{i})=\lambda_{i}-\lambda_{i_{0}}<0\text{ for
}0\not =\lambda_{i}<\lambda_{i_{0}}\text{, i.e., }i>i_{0}. \label{kappa1}%
\end{equation}
For \, $\lambda_{i_{0}}>0$, \, it\, follows \, that $\kappa(_{\mathbb{S}}\mathcal{V}%
_{c})=-\lambda_{i_{0}}<0$. This implies that $_{\mathbb{S}}\mathcal{V}^{-}$
defined \, in \, (\ref{V_inv1})  \, satisfies \, $\kappa(_{\mathbb{S}}\mathcal{V}%
^{-})$ \, $=$ \, $\lambda_{i_{1}}-\lambda_{i_{0}}<0$, \, where \, $\lambda_{i_{1}}$
$:=$ $\max\{\lambda_{\iota}\left\vert \lambda_{i}>0\text{ and }i\geq i_{0}\right.
\}$. If $\lambda_{i}<0$ for all $i>i_{0}$ then $\kappa(_{\mathbb{S}%
}\mathcal{V}^{-})=-\lambda_{i_{0}}.$

For $\lambda_{i_{0}}<0$ it holds that $\kappa(_{\mathbb{S}}\mathcal{V}%
_{c})=-\lambda_{i_{0}}>0$ and (\ref{kappa1}) implies that $_{\mathbb{S}%
}\mathcal{V}^{-}$ defined in (\ref{V_inv2}) satisfies $\kappa(_{\mathbb{S}%
}\mathcal{V}^{-})=\lambda_{\ell}-\lambda_{i_{0}}<0$ if $i_{0}<\ell$. Hence the
assertions follow.
\end{proof}

\begin{remark}
Note that $\lambda_{i}-\lambda_{i_{0}}\not =0$ for $i\not =i_{0}$ and
$\lambda_{i_{0}}\not =0$. It follows that the linear flow above is uniformly
hyperbolic, if $\dim L(\lambda_{i_{0}})=1$, i.e., if $\lambda_{i_{0}}$ is a
simple real eigenvalue of $A$.
\end{remark}

\begin{remark}
The Selgrade bundles $_{\mathbb{S}}\mathcal{V}_{i}$, $_{\mathbb{S}}%
\mathcal{V}_{c}$, and$\,_{\mathbb{S}}\mathcal{V}_{i_{0}}$ actually are
Sacker-Sell bundles; cf. Johnson, Palmer, and Sell \cite{JoPS87}. The
Sacker-Sell spectral intervals degenerate to points.
\end{remark}

\section{Invariant manifolds\label{Section5}}

In this section we determine invariant manifolds on the Poincar\'{e} sphere
$\mathbb{S}^{n}$ and on $\mathbb{R}^{n}$. In order to determine the behavior
of linear control systems of the form (\ref{linear}) \textquotedblleft near
infinity\textquotedblright\ we analyze invariant manifolds for points on the
equator $\mathbb{S}^{n,0}$ of $\mathbb{S}^{n}$. For the linearization
$T\pi\Phi^{1}$ about the invariant sets $_{\mathbb{S}}L(\lambda_{i_{0}%
})^{\infty}\subset\mathbb{S}^{n,0}$, Theorem \ref{Theorem_LinLyap} determines
the Selgrade bundles and their spectra.

We need stable manifolds for the nonlinear flow $\pi\Phi^{1}$ on
$\mathcal{U}\times\mathbb{S}^{n}$. From the vast literature on invariant
manifolds, we refer to the following related versions of stable manifold
theorems. Johnson \cite{John87} proved results on invariant manifolds
tangential to Sacker-Sell bundles. His main result \cite[Theorem 2.25]{John87}
concerns differential equations in $\mathbb{R}^{n}$ which are embedded in a
flow where the base space is of Bebutov type. Chow and Yi \cite{ChowYi}
consider differential equations in $\mathbb{R}^{n}$ with a base space which is
a compact and connected manifold. Alternatively, also invariant manifold
theorems for (single) Carath\'{e}odory differential equations are presented in
Aulbach and Wanner \cite{AulW} .

While the methods developed in these (and other) papers presumably can be
adapted so that they apply to our situation, the results are not immediately
applicable. Instead we will use a local stable manifold theorem in Colonius
and Kliemann \cite[Theorem 6.4.3]{ColK00}, which is based on an abstract
stable manifold theorem due to Bronstein and Chernii \cite{BroC} (cf.
\cite[Theorem 5.6.1]{ColK00} for a detailed proof and the monograph by
Bronstein and Kopanskii \cite{BroK94} for abstract invariant manifold theory).
We need the following notational preliminaries.

Consider a control-affine system on a Riemannian manifold in the form%
\begin{equation}
\dot{x}(t)=X_{0}(x(t))+\sum_{i=1}^{m}u_{i}(t)X_{i}(x(t)),u\in\mathcal{U},
\label{nonlinear}%
\end{equation}
with smooth vector fields $X_{i},i=0,1,\ldots,m$. We assume that for all
initial conditions $x(0)=x_{0}\in M$ and all controls $u\in\mathcal{U}$
(unique) solutions $\varphi(t,x_{0},u),t\in\mathbb{R}$, exist. The linearized
system on the tangent bundle $TM$ has the form%
\begin{equation}
\frac{d}{dt}Tx(t)=TX_{0}(Tx(t))+\sum_{i=1}^{m}u_{i}(t)TX_{i}(Tx(t)),
\label{linearized}%
\end{equation}
where for a vector field $X$ on $M$ its linearization is denoted by $TX$. The
control flow $\Phi_{t}(u,x)=(u(t+\cdot),\varphi(t,x,u))$ on $\mathcal{U}\times
M$ for (\ref{nonlinear}) can be linearized and yields the control flow $T\Phi$
on $\mathcal{U}\times TM$ for (\ref{linearized}).

The following result is a minor modification of a local stable manifold
theorem on Riemannian manifolds presented in \cite[Theorem 6.4.9]{ColK00} (a
mistake in the formulation is corrected); cf. also the local stable manifold
theorem \cite[Theorem 6.4.3]{ColK00} on $\mathbb{R}^{n}$. Instead of the
compact closure of a control set $D$ in $M$, we consider an arbitrary compact
set $K\subset M$ and define the lift of $K$ to $\mathcal{U}\times M$ by%
\[
\mathcal{K}:=\{(u,x)\in\mathcal{U}\times M\left\vert \varphi(t,x,u)\in K\text{
for all }t\in\mathbb{R}\right.  \}.
\]
We will assume that $\mathcal{K}$ is nonvoid and hence a minimal compact
invariant set for the control flow $\Phi$, and that the linearized flow
$T\Phi$ restricted to the vector bundle%
\[
\mathcal{V}_{\mathcal{K}}=\{(u,x,v)\in\mathcal{U}\times TM\left\vert
(u,x)\in\mathcal{K}\right.  \}\rightarrow\mathcal{K}%
\]
admits the following decomposition into invariant subbundles,%
\begin{equation}
\mathcal{V}_{\mathcal{K}}=\mathcal{V}^{-}%
{\textstyle\bigoplus}
\mathcal{V}^{+}. \label{BILIVEC:decomp1}%
\end{equation}
Here $\mathcal{V}^{-}$ and $\mathcal{V}^{+}$ are exponentially separated
meaning that there are constants $c_{0}>0$ and $\varepsilon_{0}>0$ with
\begin{equation}
\left\Vert T\Phi_{t}(u,x,v^{-})\right\Vert \,\leq c_{0}e^{-\varepsilon_{0}%
t}\,\left\Vert T\Phi_{t}(u,x,v^{+})\right\Vert \text{ for }t\geq0,
\label{BILIVEC:expo}%
\end{equation}
for all $(u,x,v^{-})\in\mathcal{V}^{-},$ $(u,x,v^{+})\in\mathcal{V}^{+}$ with
$\left\Vert v^{-}\right\Vert =\left\Vert v^{+}\right\Vert =1$, and the
subbundle \, $\mathcal{V}^{-}$ \, is \, stable, \, i.e., \, the \, Lyapunov \, exponents \,
$\lambda(u,x,v):=\lim\sup_{t\rightarrow\infty}\frac{1}{t}\log\left\Vert
D_{x}\varphi(t,x,u)v\right\Vert $ satisfy%
\begin{equation}
\kappa(\mathcal{V}^{-}):=\sup\left\{  \lambda(u,x,v^{-})\left\vert
(u,x,v^{-})\in\mathcal{V}^{-}\right.  \right\}  <0.
\label{BILIVEC:stablebundle}%
\end{equation}
Note that we do not assume hyperbolicity; that is, we allow that
$\mathcal{V}^{+}$ includes points with nonpositive Lyapunov exponents. For the
Riemannian manifold $M$, the exponential map $\exp$ yields a map
$(x,v)\mapsto\left(  x,\exp(x,v)\right)  $ from a neighborhood of the zero
section in $TM$ to $M\times M$. Recall the following definitions (cf. e.g.
Husemoller \cite[p. 15]{Huse75}). A bundle is a continuous map $\pi
:E\rightarrow B$, where $E$ and $B$ are metric spaces. Let $\pi:E\rightarrow
B$ and $\pi^{\prime}:E^{\prime}\rightarrow B^{\prime}$ be two bundles. A
bundle morphism $F=(g,h):E\rightarrow E^{\prime}$ is a pair of continuous maps
$g:E\rightarrow E^{\prime}$ and $h:B\rightarrow B^{\prime}$ such that
$\pi^{\prime}\circ g=h\circ\pi$. A bundle morphism is a bundle isomorphism if
there exists an inverse bundle morphism.

\begin{theorem}
\label{Theorem_LSMRiemann}Consider control system (\ref{nonlinear}) on $M$
with associated control flow $\Phi$ on $\mathcal{U}\times M$ and let $K\subset
M$ be compact with lift $\mathcal{K}$ to $\mathcal{U}\times M$. Suppose for
the linearized control system (\ref{linearized}) that the associated
linearized control flow $T\Phi$ restricted to the vector bundle $\mathcal{V}%
_{\mathcal{K}}$ admits the decomposition (\ref{BILIVEC:decomp1}),
(\ref{BILIVEC:expo}) into subbundles and the subbundle $\mathcal{V}^{-}$
satisfies stability condition (\ref{BILIVEC:stablebundle}). Then there are
$\delta>0$ and a map
\[
S^{-}:\left\{  (u,x,v^{-})\in\mathcal{V}^{-}\left\vert \left\Vert
v^{-}\right\Vert <\delta\right.  \right\}  \rightarrow\mathcal{K}\times M,
\]
which is a bundle isomorphism onto its image $\mathcal{W}_{loc}^{-}%
:=\mathrm{im\,}S^{-}$ with the following properties:

(i) Every $(u,x_{0},x)\in\mathcal{W}_{loc}^{-}$ satisfies
\[
\lim_{t\rightarrow\infty}e^{-\alpha t}d(\varphi(t,x,u),\varphi(t,x_{0}%
,u))=0\text{\ for\ every\ }\alpha\in\left(  \kappa(_{\mathbb{S}}%
\mathcal{V}^{-}),0\right)  ,
\]
where $d$ denotes the Riemannian distance on $M$. The set $\mathcal{W}%
_{loc}^{-}\subset\mathcal{K}\times M$ is called a local stable manifold
corresponding to the stable subbundle $\mathcal{V}^{-}$.

(ii) For every $(u,x_{0})\in\mathcal{K}$, define the local stable manifold for
$(u,x_{0})$ by
\[
\mathcal{W}_{loc}^{-}(u,x_{0}):=\left\{  x\in M\left\vert (u,x_{0}%
,x)\in\mathcal{W}_{loc}^{-}\right.  \right\}  \subset M.
\]
Then the topological dimension of $\mathcal{W}_{loc}^{-}(u,x_{0})$ equals the
dimension of $\mathcal{V}^{-}$.

(iii) The local stable manifold $\mathcal{W}^{-}$ is positively invariant
under the control flow $\Phi$, i.e., for $x\in\mathcal{W}_{loc}^{-}(u,x_{0})$
it follows that $\varphi(t,x,u)\in\mathcal{W}_{loc}^{-}(u(t+\cdot
),\varphi(t,x_{0},u))$\ for\ all$\;t\geq0$.

(iv) The distance of the subbundle $\mathcal{W}_{loc}^{-}$ to $\mathcal{V}%
^{-}$ can be made arbitrarily small in the following Lipschitz sense by
choosing $\delta>0$ small: For all $h>0$ there is $\delta>0$ such that
$\mathcal{W}^{-}$ is contained in the set $C(\mathcal{V}^{-},h)$ of angle $h$
around $\mathcal{V}^{-}$ given by

 \[   C(\mathcal{V}^{-},h)   \]
   \[  := \left\{  (u,x,\exp(x,v^{+}+v^{-}))\in\mathcal{K}\times
M\left\vert (u,x,v^{\pm})\in\mathcal{V}^{\pm},\left\Vert v^{+}\right\Vert \leq
h\left\Vert v^{-}\right\Vert \right.  \right\}  . \]

\end{theorem}

\begin{remark}
\label{Remark_unstable}Using time reversal, one can also obtain a result on
local unstable manifolds, cf. \cite[Remark 6.4.5]{ColK00}.
\end{remark}

We will apply Theorem \ref{Theorem_LSMRiemann} to control system
(\ref{S^n})\ on the Poincar\'{e} sphere $M=\mathbb{S}^{n}$, the associated
control flow $\pi\Phi^{1}$ on $\mathcal{U}\times\mathbb{S}^{n}$, and its
linearization $T\pi\Phi^{1}$ on $\mathcal{U}\times TS^{n}$. The set
$K:=\,_{\mathbb{S}}L(\lambda_{i_{0}})^{\infty}\subset\mathbb{S}^{n,0}$ is
compact and its lift to $\mathcal{U}\times\mathbb{S}^{n}$ is given by
$\mathcal{U}\times\,_{\mathbb{S}}L(\lambda_{i_{0}})^{\infty}$ since on the
equator $\mathbb{S}^{n,0}$ the controls $u$ act trivially. Observe that this
also implies that $\pi\varphi^{1}(t,s,u)=\pi\varphi^{1}(t,s,0)$ for
$s\in\mathbb{S}^{n,0}$. Recall that Corollary \ref{Corollary_stable_sub}
presents decompositions with a stable subbundle $_{\mathbb{S}}\mathcal{V}^{-}$
in the form%
\begin{equation}
\mathcal{V}_{_{\mathbb{S}}L(\lambda_{i_{0}})^{\infty}}=\mathcal{U}\times
T_{_{\mathbb{S}}L(\lambda_{i_{0}})^{\infty}}\mathbb{S}^{n}=\,_{\mathbb{S}%
}\mathcal{V}^{-}\oplus\,_{\mathbb{S}}\mathcal{V}^{+}. \label{decomposition2}%
\end{equation}
We obtain the following corollary to Theorem \ref{Theorem_LSMRiemann}.

\begin{corollary}
\label{Corollary_stable_man}Consider control system (\ref{S^n}) on the
Poincar\'{e} sphere $\mathbb{S}^{n}$ and the associated control flow $\pi
\Phi^{1}$ on $\mathcal{U}\times\mathbb{S}^{n}$. For the compact set
$_{\mathbb{S}}L(\lambda_{i_{0}})^{\infty}\subset\mathbb{S}^{n,0}%
,\lambda_{i_{0}}\not =0$, and the stable subbundle $_{\mathbb{S}}%
\mathcal{V}^{-}$ in (\ref{decomposition2}), there are $\delta>0$ and a map
\[
S^{-}:\left\{  (u,s,v,v_{n+1})\in\,_{\mathbb{S}}\mathcal{V}^{-}\left\vert
\left\Vert (v,v_{n+1})\right\Vert <\delta\right.  \right\}  \rightarrow
\mathcal{U}\times\,_{\mathbb{S}}L(\lambda_{i_{0}})^{\infty}\times
\mathbb{S}^{n},
\]
which is a bundle isomorphism onto its image $\mathcal{W}_{loc}^{-}%
:=\mathrm{im\,}S^{-}$, with the following properties:

(i) Every $(u,s_{0},s)\in\mathcal{W}_{loc}^{-}$ satisfies
\begin{equation}
\lim_{t\rightarrow\infty}e^{-\alpha t}d(\pi\varphi^{1}(t,s,u),\pi\varphi
^{1}(t,s_{0},u))=0\text{\ for\ every}\mathrm{\;}\alpha\in\left(
\kappa(_{\mathbb{S}}\mathcal{V}^{-}),0\right)  , \label{alpha}%
\end{equation}
where $d$ denotes the Riemannian distance on $\mathbb{S}^{n}$. The set
$\mathcal{W}_{loc}^{-}\subset\mathcal{U}\times\,_{\mathbb{S}}L(\lambda_{i_{0}%
})^{\infty}\times\mathbb{S}^{n}$ is a local stable manifold corresponding to
the stable subbundle $_{\mathbb{S}}\mathcal{V}^{-}$.

(ii) For every $(u,s_{0})\in\,_{\mathbb{S}}L(\lambda_{i_{0}})^{\infty}$ the
local stable manifold for $(u,s_{0})$ defined by
\[
\mathcal{W}_{loc}^{-}(u,s_{0}):=\left\{  s\in\mathbb{S}^{n}\left\vert
(u,s_{0},s)\in\mathcal{W}_{loc}^{-}\right.  \right\}  \subset\mathbb{S}^{n}%
\]
has topological dimension equal to the dimension of $_{\mathbb{S}}%
\mathcal{V}^{-}$.

(iii) The local stable manifold $\mathcal{W}_{loc}^{-}$ is positively
invariant under the control flow $\pi\Phi^{1}$, i.e., for $s\in\mathcal{W}%
_{loc}^{-}(u,s_{0})$ it holds that $\pi\varphi^{1}(t,s,u)\in\mathcal{W}%
_{loc}^{-}(u(t+\cdot),\pi\varphi^{1}(t,s_{0},0))$\ for\ all$\;t\geq0$.

(iv) The distance of the subbundle $\mathcal{W}_{loc}^{-}$ to $_{\mathbb{S}%
}\mathcal{V}^{-}$ can be made arbitrarily small by choosing $\delta>0$ small:
For all $h>0$ there is $\delta>0$ such that $\mathcal{W}_{loc}^{-}$ is
contained in the set $C(_{\mathbb{S}}\mathcal{V}^{-},h)$ of angle $h$ around
$_{\mathbb{S}}\mathcal{V}^{-}$ given by
\[
C(_{\mathbb{S}}\mathcal{V}^{-},h)=\left\{  (u,s,\exp(s,v^{+}+v^{-}%
))\in\mathcal{U}\times\,_{\mathbb{S}}L(\lambda_{i_{0}})^{\infty}%
\times\mathbb{S}^{n}\left\vert
\begin{array}
[c]{c}%
(u,s,v^{\pm})\in\mathcal{V}^{\pm},\\
\left\Vert v^{+}\right\Vert \,\leq h\left\Vert v^{-}\right\Vert
\end{array}
\right.  \right\}  .
\]

\end{corollary}

Note that for the sphere $\mathbb{S}^{n}$ the exponential map is
$\exp:T\mathbb{S}^{n}\rightarrow\mathbb{S}^{n}\times\mathbb{S}^{n}%
:(s,v)\mapsto(s,\exp(s,v))=\left(  s,\frac{s+v}{\left\Vert s+v\right\Vert
}\right)  .$

\begin{remark}
We can apply Theorem \ref{Theorem_LSMRiemann} also to the flow on
$\mathbb{S}^{n,0}$, which has the form $(\pi\psi(t,s),0),t\in\mathbb{R}%
,s\in\mathbb{S}^{n,0}\simeq\mathbb{S}^{n-1}\times\{0\}$; here $\pi$ denotes
the projection $\pi:\mathbb{R}_{0}^{n}\rightarrow\mathbb{S}^{n-1}$. When the
base space of the linearized flow is restricted to $_{\mathbb{S}}%
L(\lambda_{i_{0}})^{\infty}$ and $\lambda_{i_{0}}<0$, the stable subbundle is
\[
\bigoplus_{i>i_{0}}T_{_{\mathbb{S}}L(\lambda_{i_{0}})^{\infty}}(T\pi
)\,L(\lambda_{i})^{\infty}=\,_{\mathbb{S}}L(\lambda_{i_{0}})^{\infty}%
\times\bigoplus_{i>i_{0}}L(\lambda_{i})^{\infty}.
\]
One obtains a corresponding stable manifold in $\mathbb{S}^{n,0}$. If
$\lambda_{i_{0}}>0$ one has to add the subbundle $(T\pi)\,L(\lambda_{i_{0}%
})^{\infty}$, which has dimension $\dim L(\lambda_{i_{0}})-1$.
\end{remark}

Consider $(u,s_{0},s)\in\mathcal{W}_{loc}^{-}\subset\mathcal{U}\times
\mathbb{S}^{n,0}\times\mathbb{S}^{n}$. Here only the values $u(t),t\geq0$, are
relevant. We can extend $\mathcal{W}_{loc}^{-}$ to a global stable manifold
$\mathcal{W}^{-}$defined by%
\[
\mathcal{W}^{-}:=\left\{  (u,s_{0},s)\in\mathcal{U}\times\mathbb{S}%
^{n,0}\times\mathbb{S}^{n}\left\vert
\begin{array}
[c]{c}%
\exists T\geq0:\\
(u(T+\cdot),\pi\varphi^{1}(T,s_{0},0),\pi\varphi^{1}(T,s,u))\in\mathcal{W}%
_{loc}^{-}%
\end{array}
\right.  \right\}  .
\]
Note that, for $t\rightarrow-\infty$, the $\alpha$-limit set of $\pi
\varphi^{1}(t,s,u),s\in\mathbb{S}^{n,+}$, is contained in one of the central
chain control sets $_{\mathbb{S}}E_{c}^{j},j=0$ or $j=1,2$, or in one of the
sets $_{\mathbb{S}}L(\lambda_{i})_{j}^{\infty}\subset\mathbb{S}^{n,0},j=0$ or
$j=1,2,\lambda_{i}\not =0$; cf. Corollary \ref{Corollary_limit}.

The following corollary is immediate.

\begin{corollary}
In the situation of Corollary \ref{Corollary_stable_man} the global stable
manifold $\mathcal{W}^{-}$ satisfies (\ref{alpha}) for every $(u,s_{0}%
,s)\in\mathcal{W}^{-}$ and it is invariant under the control flow $\pi\Phi
^{1}$, i.e., for $s\in\mathcal{W}^{-}(u,s_{0})$ it holds that $\pi\varphi
^{1}(t,s,u)\in\mathcal{W}^{-}(u(t+\cdot),\pi\varphi^{1}(t,s_{0},0))$%
\ for\ all$\;t\in\mathbb{R}$.
\end{corollary}

Next we analyze the consequences for the original linear control system on
$\mathbb{R}^{n}$. Let $\Delta(x)=\sqrt{1+x_{1}^{2}+\cdots+x_{n}^{2}}$ and%
\begin{align}
\phi^{+}  &  :\mathbb{R}^{n}\rightarrow\mathbb{S}^{n,+},\phi^{+}%
(x)=\Delta(x)^{-1}(x_{1},\ldots,x_{n},1),\nonumber\\
\left(  \phi^{+}\right)  ^{-1}  &  :\mathbb{S}^{n,+}\rightarrow\mathbb{R}%
^{n},\left(  \phi^{+}\right)  ^{-1}(y_{1},\ldots,y_{n+1})=\left(  \frac{y_{1}%
}{y_{n+1}},\ldots,\frac{y_{n}}{y_{n+1}}\right)  . \label{phi+}%
\end{align}
This holds since $\phi^{+}$ and $\left(  \phi^{+}\right)  ^{-1}$ are bijective
and%
\[
\left(  \phi^{+}\right)  ^{-1}(\phi^{+}(x))=\left(  \phi^{+}\right)
^{-1}\left(  \frac{1}{\Delta(x)}(x_{1},\ldots,x_{n},1)\right)  =(x_{1}%
,\ldots,x_{n}).
\]
For any neighborhood $N\subset\mathbb{S}^{n}$ of $s_{0}\in\mathbb{S}^{n,0}$
one has that $\left(  \phi^{+}\right)  ^{-1}(N\cap\mathbb{S}^{n,+}%
)\not =\varnothing$. Thus for a stable manifold $\mathcal{W}_{loc}^{-}$ and
$(u,s_{0})\in\mathcal{U}\times\,_{\mathbb{S}}L(\lambda_{i_{0}})^{\infty}$, we
define%
\[
W_{loc}^{-}:=\left(  \phi^{+}\right)  ^{-1}\left(  \mathcal{W}_{loc}^{-}%
\cap\mathbb{S}^{n,+}\right)  \subset\mathbb{R}^{n},W_{loc}^{-}(u,s_{0}%
):=\left(  \phi^{+}\right)  ^{-1}\left(  \mathcal{W}^{-}(u,s_{0}%
)\cap\mathbb{S}^{n,+}\right)
\]
and analogously $W^{-}$ and $W^{-}(u,s_{0})$. For $\lambda_{i_{0}}<0$ the
stable subbundle $_{\mathbb{S}}\mathcal{V}^{-}$ given by (\ref{V_inv2}) does
not contain $_{\mathbb{S}}\mathcal{V}_{c}$, and hence $\mathcal{W}_{loc}%
^{-}(u,s_{0})$ is Lipschitz close to $\mathbb{S}^{n,0}$ (or even contained in
$\mathbb{S}^{n,0}$). Since the intersection with $\mathbb{S}^{n,+}$ is
relevant, we restrict attention to the case $\lambda_{i_{0}}>0$ where, by
(\ref{V_inv1}), the subbundle $_{\mathbb{S}}\mathcal{V}_{c}\subset
\,_{\mathbb{S}}\mathcal{V}^{-}$, and hence $\mathcal{W}^{-}(u,s_{0}%
)\cap\mathbb{S}^{n,+}$ is nonvoid.

In Corollary \ref{Corollary_stable_man}(i), the convergence of $\pi\varphi
^{1}(t,s,u)$ to $\pi\varphi^{1}(t,s_{0},u)$ means that the direction of the
corresponding trajectory $\varphi^{1}(t,s,u)$ converges to the direction of
$\varphi^{1}(t,s_{0},u)$ (the directions are given by the line through the
origin and the point on the trajectory). One easily sees that in
$\mathbb{R}^{n}$ this implies that also the direction of $\varphi(t,s,u)$
converges to the direction of $\varphi(t,s_{0},u)$.

The following example shows what may be expected for the trajectories, not
just their directions.

\begin{example}
Consider $u=0$ and let $z_{0}\in\mathbb{R}^{n}$ with $\left\Vert
z_{0}\right\Vert =1$ be\thinspace an eigenvector of $A$ for an eigenvalue
$\lambda_{i_{0}}>0$. Then $s_{0}:=(z_{0},0)\in\mathbb{S}^{n,0}$ is an
equilibrium of the induced flow on $\mathbb{S}^{n,0}$. We have
\[
\varphi(t,x,0)=e^{At}x,t\geq0,x\in\mathbb{R}^{n}\text{, and }e^{At}%
z_{0}=e^{\lambda_{i_{0}}t}z_{0},
\]
and obtain on $\mathbb{S}^{n}$%
\begin{align*}
\pi\varphi^{1}(t,\pi(z_{0},1),0)  &  =\phi^{+}(e^{At}z_{0})=\Delta(e^{At}%
z_{0})^{-1}(e^{At}z_{0},1)=\Delta(e^{\lambda_{i_{0}}t}z_{0})^{-1}%
(e^{\lambda_{i_{0}}t}z_{0},1)\\
&  \rightarrow(z_{0},0)=s_{0}.
\end{align*}
This shows that $\pi(z_{0},1)\in\mathcal{W}^{-}(0,s_{0})\cap\mathbb{S}^{n,+}$.
Define%
\[
s_{1}=\pi(z_{0},1)\text{ and }s_{2}=\pi\varphi^{1}(\tau,s_{1},0)\text{ for
some }\tau>0.
\]
It follows that $s_{1},s_{2}\in\mathcal{W}^{-}(0,s_{0})\cap\mathbb{S}^{n,+}$
due to invariance of $\mathcal{W}^{-}(0,s_{0})$ and $\mathbb{S}^{n,+}$. Define
$z_{1}:=z_{0},z_{2}=\left(  \phi^{+}\right)  ^{-1}(s_{1})=\varphi(\tau
,z_{1},0)$. The points $z_{1}$ and $z_{2}$ satisfy%
\begin{align*}
\left\Vert \varphi(t,z_{1},0)-\varphi(t,z_{2},0)\right\Vert  &  =\left\Vert
e^{At}z_{0}-e^{A(t+\tau)}z_{0}\right\Vert =\left\Vert e^{\lambda_{i_{0}}%
t}z_{0}-e^{\lambda_{i_{0}}(t+\tau)}z_{0}\right\Vert \\
&  =e^{\lambda_{i_{0}}t}\left\Vert z_{0}-e^{\lambda_{i_{0}}\tau}%
z_{0}\right\Vert .
\end{align*}
Thus the distance $\left\Vert \varphi(t,z_{1},0)-\varphi(t,z_{2},0)\right\Vert
$ grows with $e^{\lambda_{i_{0}}t}$. By stability on $\mathbb{S}^{n}$, it
follows for $\alpha\in(\kappa(_{\mathbb{S}}\mathcal{V}^{-}),0)\subset
(-\infty,0)$,%
\begin{align*}
&  d(\pi\varphi^{1}(t,s_{1},0),\pi\varphi^{1}(t,s_{2},0))\\
&  \leq d(\pi\varphi^{1}(t,s_{1},0),\pi\varphi^{1}(t,s_{0},0))+d(\pi
\varphi^{1}(t,s_{0},0),\pi\varphi^{1}(t,s_{2},0))\\
&  \leq e^{-\alpha t}\left[  d(s_{1},s_{0})+d(s_{0},s_{2})\right]
\rightarrow0\text{ for }t\rightarrow\infty.
\end{align*}
For $t\rightarrow\infty$, the points on $\mathcal{W}^{-}(0,s_{0}%
)\cap\mathbb{S}^{n,+}$ converge exponentially, while the points on the
preimage $W^{-}(0,s_{0})=\left(  \phi^{+}\right)  ^{-1}\left(  \mathcal{W}%
^{-}(0,s_{0})\cap\mathbb{S}^{n,+}\right)  \subset\mathbb{R}^{n}$ diverge
exponentially with $e^{\lambda_{i_{0}}t}$. Note that for $t^{\prime
}(t):=t-\tau,t\geq\tau$, one obtains $\left\Vert \varphi(t,z_{1}%
,0)-\varphi(t^{\prime}(t),z_{2},0)\right\Vert =0$.
\end{example}

In the general situation, we obtain the following result.

\begin{theorem}
\label{Theorem_R}Let the assumptions of Corollary \ref{Corollary_stable_man}
be satisfied and fix $s_{0}\in\,_{\mathbb{S}}L(\lambda_{i_{0}})^{\infty
}\subset\mathbb{S}^{n,0}$ with $\lambda_{i_{0}}>0$.

(i) For every $(u,s_{0})\in\mathcal{U}\times\,_{\mathbb{S}}L(\lambda_{i_{0}%
})^{\infty}$, the dimension of $W_{loc}^{-}(u,s_{0})\subset\mathbb{R}^{n}$
satisfies $\dim W_{loc}^{-}(u,s_{0})\leq\dim$ $_{\mathbb{S}}\mathcal{V}^{-}$.
For all $z\in W^{-}(u,s_{0})$ it holds that $\left\Vert \varphi
(t,z,u)\right\Vert \rightarrow\infty$ for $t\rightarrow\infty$ and the sets
$W^{-}(u,s_{0})$ are invariant in the sense that%
\[
\varphi(t,z,u)\in W^{-}(u(t+\cdot),\pi\varphi^{1}(t,s_{0},u))\text{ for all
}t\in\mathbb{R}.
\]

(ii) For $j=1,2$, let $z_{j}:=\left(  \phi^{+}\right)  ^{-1}(s_{j})$ with
$s_{j}\in\mathcal{W}^{-}(u,s_{0})\cap\mathbb{S}^{n,+}$. Then, for $t>0$ large
enough, there is $t^{\prime}(t)>0$ with $t^{\prime}(t)\rightarrow\infty$ for
$t\rightarrow\infty$ such that%
\[
\left\Vert \varphi(t,z_{1},u)-\varphi(t^{\prime}(t),z_{2},u)\right\Vert
\leq\left\Vert \varphi^{1}(t,z_{1},1,u)\right\Vert d(\pi\varphi^{1}%
(t,s_{1},u),\pi\varphi^{1}(t^{\prime}(t),s_{2},u)).
\]

\end{theorem}

\begin{proof}
(i) The invariance property follows from invariance of $\mathcal{W}%
^{-}(u,s_{0})$ and $\mathbb{S}^{n,+}$ and the conjugacy property of $\phi^{+}%
$. The assertion $\dim W_{loc}^{-}(u,s_{0})\leq\dim$ $_{\mathbb{S}}%
\mathcal{V}^{-}$ holds since $\phi^{+}$ is a diffeomorphism and $\phi
^{+}(W_{loc}^{-}(u,s_{0}))\subset\mathcal{W}_{loc}^{-}(u,s_{0})$. Furthermore,
$\left\Vert \varphi(t,z,u)\right\Vert \rightarrow\infty$ since $\phi
^{+}(\varphi(t,z,u))\rightarrow s_{0}\in\mathbb{S}^{n,0}$.

(ii) Corollary \ref{Corollary_stable_man} and the inequality%
\begin{align*}
&  d(\pi\varphi^{1}(t,s_{1},u),\pi\varphi^{1}(t,s_{2},u))\\
&  \leq d(\pi\varphi^{1}(t,s_{1},u),\pi\varphi^{1}(t,s_{0},u))+d(\pi
\varphi^{1}(t,s_{0},u),\pi\varphi^{1}(t,s_{2},u)),
\end{align*}
imply that, for\ every\ $\alpha\in\left(  \kappa(_{\mathbb{S}}\mathcal{V}%
^{-}),0\right)  $,
\begin{equation}
\lim_{t\rightarrow\infty}e^{-\alpha t}d(\pi\varphi^{1}(t,s_{1},u),\pi
\varphi^{1}(t,s_{2},u))=0. \label{limit1}%
\end{equation}
For $j=1,2$, formula (\ref{phi+}) shows that the components of $\pi\varphi
^{1}(t,s_{j},u)$ satisfy, for $t\geq0$,%
\[
\varphi(t,z_{j},u)=\left(  \phi^{+}\right)  ^{-1}\left(  \pi\varphi
^{1}(t,s_{j},u)\right)  =\left(  \pi\varphi^{1}(t,s_{j},u)_{n+1}\right)
^{-1}\left(  \pi\varphi^{1}(t,s_{j},u)_{i}\right)  _{i=1}^{n}.
\]
With $\pi(z_{j},1)=s_{j}$, we have%
\[
\pi\varphi^{1}(t,s_{j},u)_{n+1}=\frac{\varphi^{1}(t,z_{j},1,u)_{n+1}%
}{\left\Vert \varphi^{1}(t,z_{j},1,u)\right\Vert }=\frac{1}{\left\Vert
\varphi^{1}(t,z_{j},1,u)\right\Vert }%
\]
implying%
\begin{equation}
\varphi(t,z_{j},u)=\left\Vert \varphi^{1}(t,z_{j},1,u)\right\Vert \left(
\pi\varphi^{1}(t,\pi(z_{j},1),u)_{i}\right)  _{i=1}^{n}. \label{s_z}%
\end{equation}
We claim that, for $t>0$ large enough, there is $t^{\prime}(t)>0$ with
$\left\Vert \varphi^{1}(t,z_{1},1,u)\right\Vert =\left\Vert \varphi
^{1}(t^{\prime}(t),z_{2},1,u)\right\Vert $. In fact, we know that $\left\Vert
\varphi^{1}(t,z_{j},1,u)\right\Vert \rightarrow\infty$ for $t\rightarrow
\infty,j=1,2$. This implies that $\left\Vert \varphi^{1}(t,z_{1}%
,1,u)\right\Vert >\left\Vert (z_{2},1)\right\Vert $ for $t$ large enough, and
hence $t^{\prime}(t)$ exists. Thus%
\begin{align*}
&  \left\Vert \varphi(t,z_{1},u)-\varphi(t^{\prime}(t),z_{2},u)\right\Vert \\
&  =\left\Vert \left\Vert \varphi^{1}(t,z_{1},1,u)\right\Vert \left(
\pi\varphi^{1}(t,s_{1},u)\right)  _{i=1}^{n}\right.  \left.  -\left\Vert
\varphi^{1}(t^{\prime}(t),z_{2},1,u)\right\Vert \left(  \pi\varphi
^{1}(t^{\prime}(t),s_{2},1),u)\right)  _{i=1}^{n}\right\Vert \\
&  \leq\left\Vert \varphi^{1}(t,z_{1},1,u)\right\Vert d(\pi\varphi^{1}%
(t,s_{1},u),\pi\varphi^{1}(t^{\prime}(t),s_{2},u)).
\end{align*}

\end{proof}

\begin{remark}
Note that%
\begin{align*}
&  d(\pi\varphi^{1}(t,s_{1},u),\pi\varphi^{1}(t^{\prime}(t),s_{2},u))\\
&  \leq d(\pi\varphi^{1}(t,s_{1},u),\pi\varphi^{1}(t,s_{2},u))+d(\pi
\varphi^{1}(t,s_{2},u),\pi\varphi^{1}(t^{\prime}(t),s_{2},u)).
\end{align*}
We know that $\lim_{t\rightarrow\infty}e^{\alpha t}\left\Vert \pi\varphi
^{1}(t,s_{1},u)-\pi\varphi^{1}(t,s_{2},u)\right\Vert =0$ for $\alpha\in\left(
\kappa(_{\mathbb{S}}\mathcal{V}^{-}),0\right)  $. If the exponential growth of
$\left\Vert \varphi^{1}(t,z_{1},1,u)\right\Vert $ is smaller than
$-\kappa(_{\mathbb{S}}\mathcal{V}^{-})$ one obtains for the first summand that%
\[
\lim_{t\rightarrow\infty}e^{\alpha t}\left\Vert \varphi^{1}(t,z_{1}%
,1,u)\right\Vert \left\Vert \pi\varphi^{1}(t,s_{1},u)-\pi\varphi^{1}%
(t,s_{2},u)\right\Vert =0.
\]
If $s_{0}$ is not an equilibrium for $\pi\varphi^{1}(\cdot,s_{0},u)$, it may
happen that in the second summand $\pi\varphi^{1}(t,s_{2},u)$ and $\pi
\varphi^{1}(t^{\prime}(t),s_{2},u)$ do not converge for $t\rightarrow\infty$.
In any case, there is $c>0$ such that $\left\Vert \pi\varphi^{1}%
(t,s_{2},u) \right.$ $-$ $\left. \pi\varphi^{1}(t^{\prime}(t),s_{2},u)\right\Vert \leq c$.
Furthermore, compactness of $_{\mathbb{S}}L(\lambda_{i_{0}})^{\infty}$ implies
that for every sequence $t_{k}\rightarrow\infty$ there are $s_{2}^{\prime
},s_{2}^{\prime\prime}\in\,_{\mathbb{S}}L(\lambda_{i_{0}})^{\infty}$ and a
subsequence $t_{k_{i}},i\in\mathbb{N}$, such that%
\[
\pi\varphi^{1}(t_{k_{i}},s_{2},u)\rightarrow s_{2}^{\prime}\text{ and }%
\pi\varphi^{1}(t^{\prime}(t_{k_{i}}),s_{2},u)\rightarrow s_{2}^{\prime\prime
}.
\]

\end{remark}

\begin{remark}
\label{Remark_unstable2}Using time reversal, one can show that the results
derived in this section for stable manifolds have counterparts for unstable
manifolds (cf. Remark \ref{Remark_unstable}). Thus one obtains invariant
manifolds $W^{+}(u,s_{0})$ in $\mathbb{R}^{n}$ consisting of points with
$\left\Vert \varphi(t,z,u)\right\Vert \rightarrow\infty$ for $t\rightarrow
-\infty$. We omit the details.
\end{remark}

\section{Examples\label{Section6}}

In this section we present several examples illustrating the results in the
previous sections.

The following two-dimensional hyperbolic system has been analyzed in Colonius,
Santana, Setti \cite[Example 2]{ColSS24a} and Colonius, Santana, Viscovini
\cite[Example 6.2]{ColSV24},%
\begin{equation}
\left(
\begin{array}
[c]{c}%
\dot{x}_{1}\\
\dot{x}_{2}%
\end{array}
\right)  =\left(
\begin{array}
[c]{cc}%
1 & 0\\
0 & -1
\end{array}
\right)  +\left(
\begin{array}
[c]{c}%
1\\
1
\end{array}
\right)  u(t)\text{ with }u(t)\in\lbrack-1,1]. \label{Example1}%
\end{equation}
For the induced system on the northern hemisphere $\mathbb{S}^{2,+}$ of the
Poincar\'{e} sphere $\mathbb{S}^{2}$, one obtains an asymptotically stable
equilibrium $s_{0}=(1,0,0)\subset\mathbb{S}^{2,0}\simeq\mathbb{S}^{1}$ and an
unstable equilibrium $s_{0}=(0,1,0)$. The phase portrait on $\mathbb{S}^{2,+}$
is sketched in \cite[Fig. 2]{ColSS24a}.

We turn to describe higher dimensional examples, where the matrix $A$ is
hyperbolic and given in different Jordan normal forms. In the classical
treatise Arnol'd \cite[\S \ 21]{ArnV92} one finds graphical illustrations for
the corresponding phase portraits in $\mathbb{R}^{3}$. For hyperbolic matrix
$A$, Theorem \ref{Theorem_cs} implies that the unique chain control set $E$ in
$\mathbb{R}^{n}$ is bounded and coincides with the closure of the control set
$D_{0}$ containing the origin. 

\begin{example}
Consider the system in $\mathbb{R}^{3}$ given by%
\begin{equation}
\left(
\begin{array}
[c]{c}%
\dot{x}_{1}\\
\dot{x}_{2}\\
\dot{x}_{3}%
\end{array}
\right)  =\left(
\begin{array}
[c]{ccc}%
2 & 0 & 0\\
0 & 1 & 0\\
0 & 0 & -1
\end{array}
\right)  \left(
\begin{array}
[c]{c}%
x_{1}\\
x_{2}\\
x_{3}%
\end{array}
\right)  +\left(
\begin{array}
[c]{c}%
1\\
1\\
1
\end{array}
\right)  u(t) , \text{ with }u(t)\in\lbrack-1,1]. \label{Example2}%
\end{equation}
The Lyapunov spaces are, with $\lambda_{1}=2,\lambda_{2}=1$, and $\lambda
_{3}=-1$, given by%
\[
L(2)=\mathbb{R}\times\{0_{2}\} , L(1)=\left\{  0\right\}  \times\mathbb{R}%
\times\left\{  0\right\}  ,\text{and }L(-1)=\{0_{2}\}\times\mathbb{R}.
\]
For $\lambda_{i_{0}}=\lambda_{2}=1>0$ consider the set $_{\mathbb{S}%
}L(1)^{\infty}=\left\{  (0,1,0,0)\right\}  \subset\mathbb{S}^{3,0}%
\simeq\mathbb{S}^{2}$. The point $s_{0}=(0,1,0,0)$ is an equilibrium at
infinity. According to Theorem \ref{Theorem_LinLyap}, for the linearized flow
$T\pi\Phi^{1}$ with base space restricted to $\mathcal{U}\times\,_{\mathbb{S}%
}L(1)^{\infty}$, the one-dimensional subbundle $_{\mathbb{S}}\mathcal{V}%
_{c}\subset T_{_{\mathbb{S}}L(1)^{\infty}}\mathbb{S}^{3}=T_{(0,1,0,0)}%
\mathbb{S}^{3}$ is stable with $\kappa(_{\mathbb{S}}\mathcal{V}_{c}%
)=-\lambda_{i_{0}}=-1$. The central subbundle $\mathcal{V}_{c}$ is determined
by the unique bounded solutions $e(u,\cdot),u\in\mathcal{U}$. The stable
subbundle is $_{\mathbb{S}}\mathcal{V}^{-}=\,_{\mathbb{S}}\mathcal{V}%
_{1}\oplus\,_{\mathbb{S}}\mathcal{V}_{c}$ where%
\begin{align*}
_{\mathbb{S}}\mathcal{V}_{1}  &  =(\mathrm{id}_{\mathcal{U}},T\pi)\left(
\mathcal{U}\times L(\lambda_{1})^{\infty}\right)  =\mathcal{U}\times
\,_{\mathbb{S}}L(1)^{\infty}\times L(2)^{\infty}\\
&  =\mathcal{U}\times\left\{  (0,1,0,0)\right\}  \times\left(  \mathbb{R}%
\times\{0_{2}\}\right)  \subset\mathcal{U}\times T\mathbb{S}^{3,0}.
\end{align*}
The unstable subbundle is
\[
_{\mathbb{S}}\mathcal{V}_{2}=(\mathrm{id}_{\mathcal{U}},T\pi)\left(
\mathcal{U}\times L(\lambda_{3})^{\infty}\right)  =\mathcal{U}\times\left\{
(0,1,0,0)\right\}  \times\left(  \{0_{2}\}\times\mathbb{R}^{2}\right)  .
\]
By Corollary \ref{Corollary_stable_man} for $u\in\mathcal{U}$, the local
stable manifold $\mathcal{W}_{loc}^{-}(u,s_{0})\subset\mathbb{S}^{3}$ is
two-dimensional and every point $s\in\mathcal{W}^{-}(u,s_{0})$ satisfies for
all $\alpha>-1$%
\[
e^{-\alpha t}d\left(  \pi\varphi^{1}(t,s,u),\pi\varphi^{1}(t,s_{0},u)\right)
=e^{-\alpha t}d\left(  \pi\varphi^{1}(t,s,u),s_{0}\right)  \rightarrow0\text{
for }t\rightarrow\infty.
\]
For all $u\in\mathcal{U}$ and all $x\in W^{-}(u,s_{0})=\left(  \phi
^{+}\right)  ^{-1}\left(  \mathcal{W}^{-}(u,s_{0})\cap\mathbb{S}^{3,+}\right)
$ it holds that $\left\Vert \varphi(t,x,u)\right\Vert \rightarrow\infty$ for
$t\rightarrow\infty$.

The local unstable manifold $\mathcal{W}_{loc}^{+}(u,s_{0})\subset
\mathbb{S}^{3}$ (cf. Remark \ref{Remark_unstable2}) is one-dimensional, and
for all $u\in\mathcal{U}$ and all $x\in W^{+}(u,s_{0})=\left(  \phi
^{+}\right)  ^{-1}\left(  \mathcal{W}^{+}(u,s_{0})\cap\mathbb{S}^{3,+}\right)
$ it holds that $\left\Vert \varphi(t,x,u)\right\Vert \rightarrow\infty$ for
$t\rightarrow-\infty$.

For the system with $u(t)\equiv0$, Fig. \ref{Fig1} sketches the phase
portraits on $\mathbb{R}^{3}$ and on $\mathbb{S}^{2}$ (we cannot draw the
phase portrait on $\mathbb{S}^{3}$). The sphere $\mathbb{S}^{2}$ is identified
with the equator $\mathbb{S}^{3,0}$ of the Poincar\'{e} sphere $\mathbb{S}%
^{3}$, hence it represents infinity. There are four equilibria at infinity on
the equator $\mathbb{S}^{2,0}$ of $\mathbb{S}^{2}$ and the poles $(0,0,\pm1)$
of $\mathbb{S}^{2}$ are unstable equilibria at infinity. The equilibrium
$s_{0}=(0,1,0,0)\in\mathbb{S}^{3,0}$ is identified with $(0,1,0)\in
\mathbb{S}^{2,0}\subset\mathbb{S}^{2}$. Its stable and unstable manifolds on
$\mathbb{S}^{2}$ are the half circle between $(0,0,1)$ and $(0,0,-1)$ and the
half circle between $(1,0,0)$ and $(-1,0,0)$, respectively. The set
$W^{-}(0,s_{0})$ is contained in the half-space in $\mathbb{R}^{3}$ spanned by
$(0,1,0)$ and $(0,0,1)$.
\end{example}

In our next example the set $_{\mathbb{S}}L(1)^{\infty}$ is not a minimal set
for the induced flow on $\mathbb{S}^{3,0}$.

\begin{example}
Consider%
\begin{equation}
\left(
\begin{array}
[c]{c}%
\dot{x}_{1}\\
\dot{x}_{2}\\
\dot{x}_{3}%
\end{array}
\right)  =\left(
\begin{array}
[c]{ccc}%
1 & 1 & 0\\
0 & 1 & 0\\
0 & 0 & -1
\end{array}
\right)  \left(
\begin{array}
[c]{c}%
x_{1}\\
x_{2}\\
x_{3}%
\end{array}
\right)  +\left(
\begin{array}
[c]{c}%
1\\
1\\
1
\end{array}
\right)  u(t)\text{ with }u(t)\in\lbrack-1,1]. \label{Example3}%
\end{equation}
The Lyapunov spaces are, with $\lambda_{1}=1$ and $\lambda_{2}=-1$ given by%
\[
L(1)=\mathbb{R}^{2}\times\left\{  0\right\}  \text{ and }L(-1)=\{0_{2}%
\}\times\mathbb{R}.
\]
For $\lambda_{i_{0}}=\lambda_{1}=1>0$ consider the set%
\[
_{\mathbb{S}}L(1)^{\infty}=\left\{  (s_{1},s_{2},0,0)\left\vert s_{1}%
^{2}+s_{2}^{2}=1\right.  \right\}  \subset\mathbb{S}^{3,0}.
\]
When we again identify the equator $\mathbb{S}^{3,0}$ with $\mathbb{S}^{2}$,
the set $_{\mathbb{S}}L(1)^{\infty}$ is identified with the equator
$\mathbb{S}^{2,0}\simeq\mathbb{S}^{1}$ of $\mathbb{S}^{2}$. For the linearized
flow $T\pi\Phi^{1}$ with base space restricted to $\mathcal{U}\times
\,_{\mathbb{S}}L(1)^{\infty}$, the one-dimensional subbundle $_{\mathbb{S}%
}\mathcal{V}_{c}\subset T_{_{\mathbb{S}}L(1)^{\infty}}\mathbb{S}^{3}$ is
stable with $\kappa(_{\mathbb{S}}\mathcal{V}_{c})=-\lambda_{i_{0}}=-1$.
Furthermore, also the subbundle $_{\mathbb{S}}\mathcal{V}_{i_{0}%
}=\,_{\mathbb{S}}\mathcal{V}_{1}$ is one-dimensional and the corresponding
Lyapunov exponents are $0$. The subbundle $_{\mathbb{S}}\mathcal{V}_{2}$ is
stable with Lyapunov exponents equal to $\lambda_{2}-\lambda_{i_{0}}=-1-1=-2$.
Hence the set $\mathcal{U}\times\,_{\mathbb{S}}L(1)^{\infty}$ is stable. For
the system with $u(t)\equiv0$, Fig. \ref{Fig2} sketches the phase portraits on
$\mathbb{R}^{3}$ and on $\mathbb{S}^{2}\simeq\mathbb{S}^{3,0}$. There are two
equilibria $(\pm1,0,0)$ at infinity on the equator $\mathbb{S}^{2,0}$ of
$\mathbb{S}^{2}$ which correspond to the eigenspace $\mathbb{R}\times\{0\}$ of
$\left(
\begin{array}
[c]{cc}%
1 & 1\\
0 & 1
\end{array}
\right)  $. Furthermore, the poles $(0,0,\pm1)$ of $\mathbb{S}^{2}$ are
unstable equilibria at infinity.
\end{example}

In the following example the matrix $A$ has a complex-conjugate pair of
eigenvalues. Here a periodic solution at infinity is obtained.

\begin{example}
Consider%
\begin{equation}
\left(
\begin{array}
[c]{c}%
\dot{x}_{1}\\
\dot{x}_{2}\\
\dot{x}_{3}%
\end{array}
\right)  =\left(
\begin{array}
[c]{ccc}%
1 & 1 & 0\\
-1 & 1 & 0\\
0 & 0 & -1
\end{array}
\right)  \left(
\begin{array}
[c]{c}%
x_{1}\\
x_{2}\\
x_{3}%
\end{array}
\right)  +\left(
\begin{array}
[c]{c}%
1\\
1\\
1
\end{array}
\right)  u(t)\text{ with }u(t)\in\lbrack-1,1]. \label{Example4}%
\end{equation}
The eigenvalues of the matrix $A$ are $\mu_{1,2}=1\pm\imath$ and $\mu_{3}=-1$,
hence the Lyapunov exponents are $\lambda_{1}=1$ and $\lambda_{2}=-1$ with
Lyapunov spaces%
\[
L(1)=\mathbb{R}^{2}\times\{0\}\text{ and }L(-1)=\{0_{2}\}\times\mathbb{R}.
\]
For $\lambda_{i_{0}}=\lambda_{1}=1>0$ consider the set%
\[
_{\mathbb{S}}L(1)^{\infty}=\left\{  (s_{1},s_{2},0,0)\left\vert s_{1}%
^{2}+s_{2}^{2}=1\right.  \right\}  \subset\mathbb{S}^{3,0}.
\]
This set identified with the equator $\mathbb{S}^{2,0}\simeq\mathbb{S}^{1}$ of
$\mathbb{S}^{2}\simeq\mathbb{S}^{3,0}$. Restricted to $_{\mathbb{S}%
}L(1)^{\infty}$ the orbit $\pi\varphi^{1}(\cdot,s,0)$ corresponds to
\[
\dot{x}_{1}(t)=x_{2}(t),\dot{x}_{2}(t)=-x_{1}(t)\text{, hence }\left(
x_{1}(t),x_{2}(t)\right)  =(\sin t,\cos t).
\]
For the linearized flow $T\pi\Phi^{1}$ with base space restricted to
$\mathcal{U}\times\,_{\mathbb{S}}L(1)^{\infty}$, the subbundle $_{\mathbb{S}%
}\mathcal{V}_{c}\subset T_{_{\mathbb{S}}L(1)^{\infty}}\mathbb{S}^{3}$ is
stable with $\kappa(_{\mathbb{S}}\mathcal{V}_{c})=-\lambda_{i_{0}}=-1$. The
subbundle $_{\mathbb{S}}\mathcal{V}_{i_{0}}=\,_{\mathbb{S}}\mathcal{V}_{1}$ is
one-dimensional. and the subbundle $_{\mathbb{S}}\mathcal{V}_{2}$ is stable,
hence the set $\mathcal{U}\times\,_{\mathbb{S}}L(1)^{\infty}$, i.e., the
periodic solution at infinity, is stable. For the system with $u(t)\equiv0$,
Fig. \ref{Fig3} sketches the phase portraits on $\mathbb{R}^{3}$ and on
$\mathbb{S}^{2}\simeq\mathbb{S}^{3,0}$.
\end{example}

In the next example, the periodic solution at infinity has a nontrivial stable manifold.

\begin{example}
Consider%
\begin{equation}
\left(
\begin{array}
[c]{c}%
\dot{x}_{1}\\
\dot{x}_{2}\\
\dot{x}_{3}\\
\dot{x}_{4}%
\end{array}
\right)  =\left(
\begin{array}
[c]{cccc}%
2 & 0 & 0 & 0\\
0 & 1 & 1 & 0\\
0 & -1 & 1 & 0\\
0 & 0 & 0 & -1
\end{array}
\right)  \left(
\begin{array}
[c]{c}%
x_{1}\\
x_{2}\\
x_{3}\\
x_{4}%
\end{array}
\right)  +\left(
\begin{array}
[c]{c}%
1\\
1\\
1\\
1
\end{array}
\right)  u(t)\text{ with }u(t)\in\lbrack-1,1]. \label{Example5}%
\end{equation}
The eigenvalues of the matrix $A$ are $\mu_{1}=2,\mu_{2,3}=1\pm\imath$, and
$\mu_{4}=-1$, hence the Lyapunov exponents are $\lambda_{1}=2,$ $\lambda
_{2}=1$, and $\lambda_{3}=-1$ with Lyapunov spaces%
\[
L(2)=\mathbb{R}\times\left\{  0_{3}\right\}  ,L(1)=\{0\}\times\mathbb{R}%
^{2}\times\{0\}\text{ and }L(-1)=\{0_{3}\}\times\mathbb{R}.
\]
For $\lambda_{i_{0}}=\lambda_{2}=1>0$ consider the set%
\[
_{\mathbb{S}}L(1)^{\infty}=\left\{  (0,s_{2},s_{3},0,0)\left\vert s_{2}%
^{2}+s_{3}^{2}=1\right.  \right\}  \subset\mathbb{S}^{4,0}.
\]
This set is identified with the subset $\{0\}\times\mathbb{S}^{2,0}%
\times\{0\}\simeq\mathbb{S}^{1}$ of $\mathbb{S}^{3}\simeq\mathbb{S}^{4,0}$.
Restricted to $_{\mathbb{S}}L(1)^{\infty}$ the trajectory $\pi\varphi
^{1}(\cdot,s,0)$ corresponds to
\[
\dot{x}_{2}(t)=x_{3}(t),\dot{x}_{3}(t)=-x_{2}(t)\text{, hence }\left(
x_{2}(t),x_{3}(t)\right)  =(\sin t,\cos t).
\]
For the linearized flow $T\pi\Phi^{1}$ with base space restricted to
$\mathcal{U}\times\,_{\mathbb{S}}L(1)^{\infty}$, the subbundle $_{\mathbb{S}%
}\mathcal{V}_{c}\subset T_{_{\mathbb{S}}L(1)^{\infty}}\mathbb{S}^{4}$ is
stable with $\kappa(_{\mathbb{S}}\mathcal{V}_{c})=-\lambda_{i_{0}}=-1$. For
points in the subbundles $_{\mathbb{S}}\mathcal{V}_{1}$ and $_{\mathbb{S}%
}\mathcal{V}_{3}$ the Lyapunov exponents are%
\[
\lambda_{1}-\lambda_{i_{0}}=2-1=1\text{ and }\lambda_{3}-\lambda_{i_{0}%
}=-1-1=-2,
\]
respectively. Hence the stable subbundle is $_{\mathbb{S}}\mathcal{V}%
_{c}\oplus\,_{\mathbb{S}}\mathcal{V}_{3}$ and $_{\mathbb{S}}\mathcal{V}_{1}$
is unstable. Let $s_{0}\in\,_{\mathbb{S}}L(1)^{\infty}$ and $u\in\mathcal{U}$.
The local stable manifold $\mathcal{W}_{loc}^{-}(u,s_{0})\subset\mathbb{S}%
^{4}$ is two-dimensional and every point $s\in\mathcal{W}^{-}(u,s_{0})$
satisfies for all $\alpha\in(-1,0)$%
\[
e^{-\alpha t}d\left(  \pi\varphi^{1}(t,s,u),\pi\varphi^{1}(t,s_{0},0)\right)
\rightarrow0\text{ for }t\rightarrow\infty.
\]
For all $x\in W^{\pm}(u,s_{0})=\left(  \phi^{+}\right)  ^{-1}\left(
\mathcal{W}^{\pm}(u,s_{0})\cap\mathbb{S}^{4,+}\right)  $ it holds that
$\left\Vert \varphi(t,x,u)\right\Vert \rightarrow\infty$ for $t\rightarrow
\pm\infty$.
\end{example}

\section{Declarations}

Funding:

Second author is partially supported by CNPq grant n. 309409/2023-3.

\vspace{0,2cm}

Conflict of interest/Competing interests:

No interests of a financial or personal nature.

\vspace{0,2cm}

Ethics approval:

Not applicable.

\vspace{0,2cm}

Use of AI:

We do not use it anywhere.

\vspace{0,2cm}

Authors' contributions:

All authors contributed to all sections. All authors reviewed the final manuscript.

\begin{figure}[t]
\centering
\includegraphics[scale=0.1]{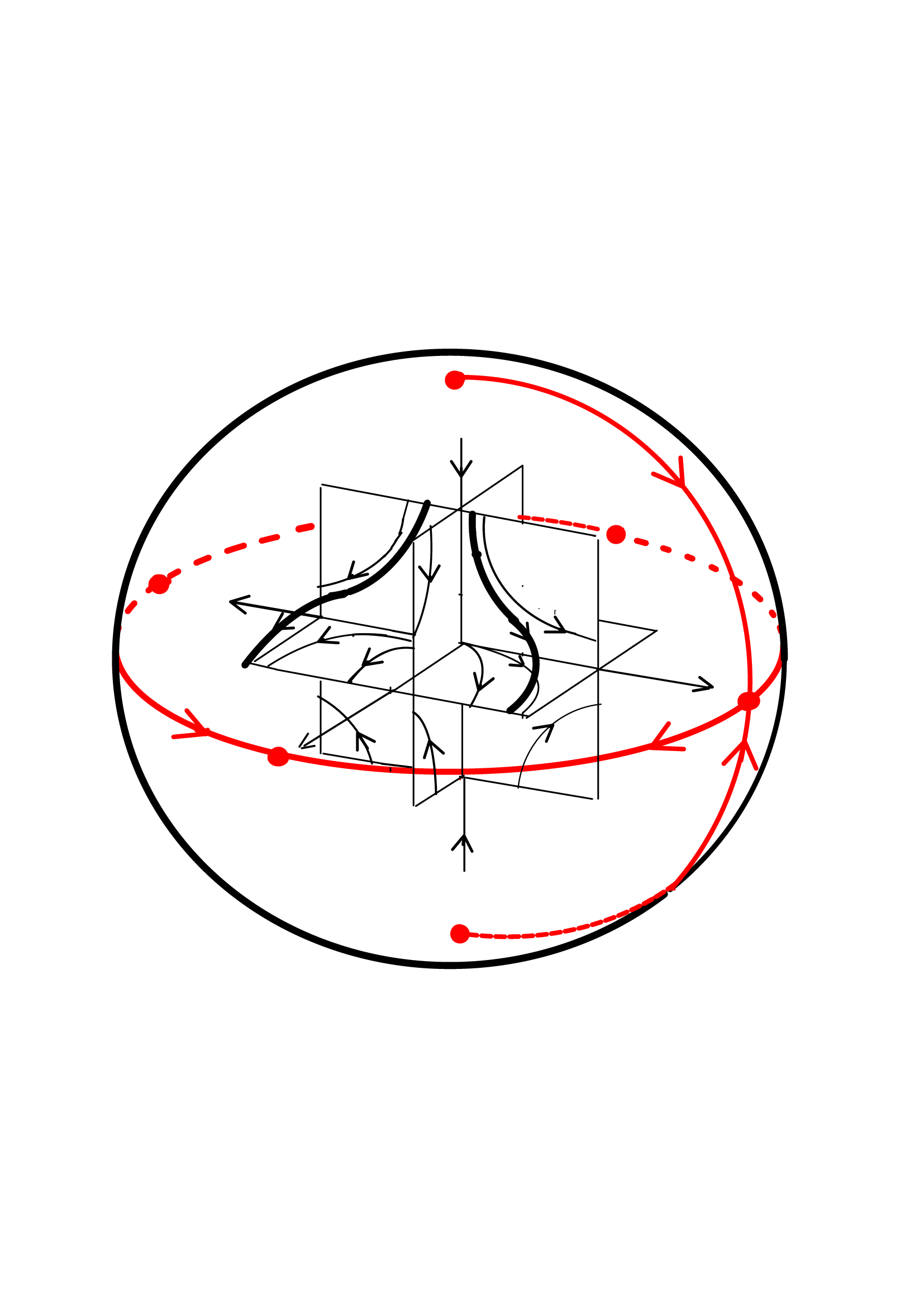}
\caption{ Phase portraits for $u(t)\equiv0$ in $\mathbb{R}^{3}$ and
$\mathbb{S}^{2}\simeq\mathbb{S}^{3,0}$ .}
\label{Fig1}%
\end{figure}

\begin{figure}[t]
\centering
\includegraphics[scale=0.1]{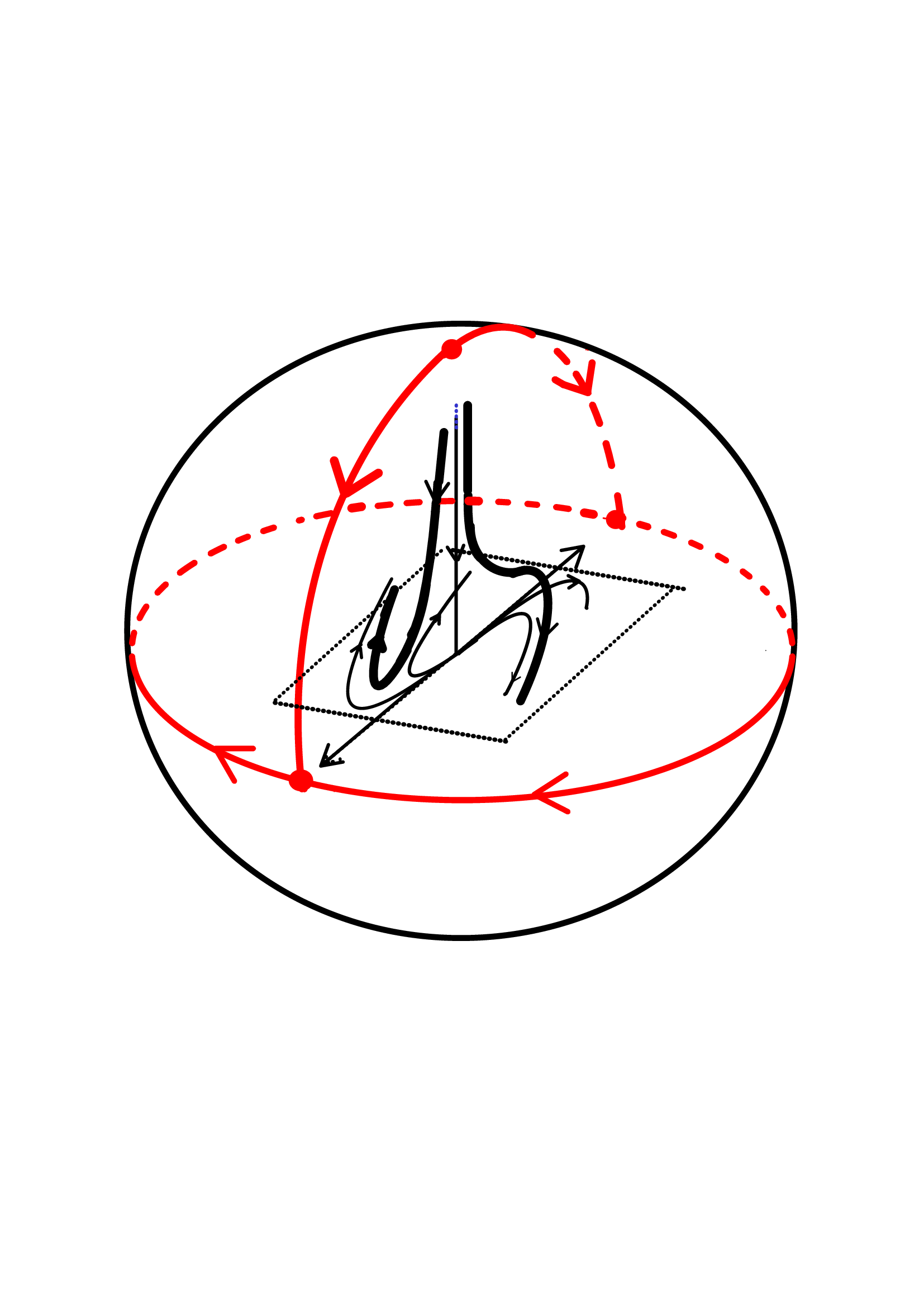}
\caption{  Phase portraits for $u(t)\equiv0$ in $\mathbb{R}^{3}$ and
$\mathbb{S}^{2}\simeq\mathbb{S}^{3,0}$ .}
\label{Fig2}%
\end{figure}

\begin{figure}[t]
\centering
\includegraphics[scale=0.1]{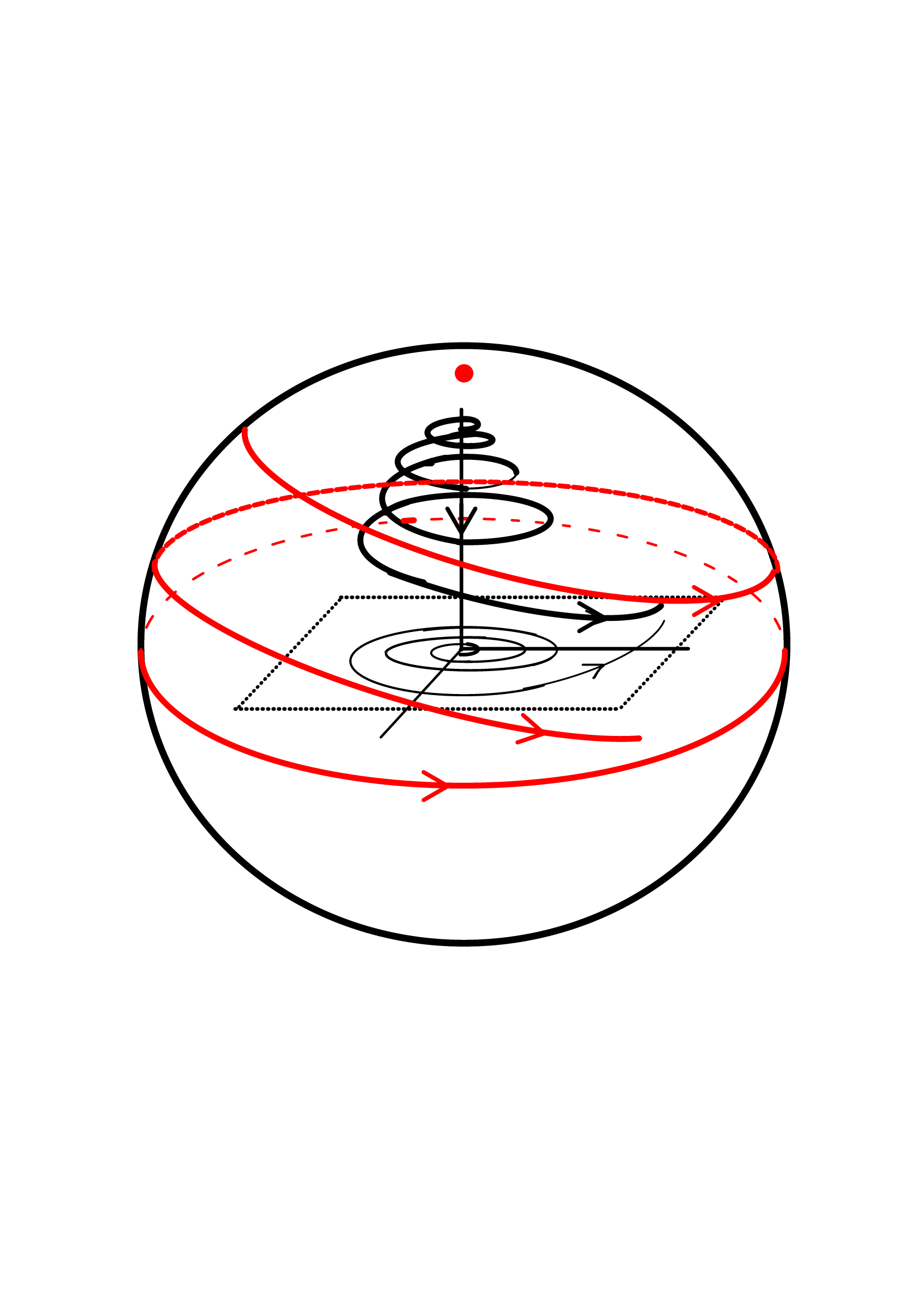}
\caption{ Phase portraits for $u(t)\equiv0$ in $\mathbb{R}^{3}$ and
$\mathbb{S}^{2}\simeq\mathbb{S}^{3,0}$.}
\label{Fig3}%
\end{figure}

\end{document}